\documentclass[a4paper,12pt]{amsart}
\usepackage[T1]{fontenc}
\usepackage[centering,margin=2.3cm]{geometry}
\usepackage[latin1]{inputenc}
\usepackage{amsthm, amsmath}   
\usepackage{latexsym,amssymb}
\usepackage{algorithmic}
\usepackage{amssymb}
\usepackage{graphicx,subfigure}
\usepackage{times}
\usepackage{enumerate}
\usepackage[usenames,dvipsnames]{pstricks}
\usepackage{epsfig}
\usepackage{pst-node}
\usepackage{pst-grad} 
\usepackage{pst-plot} 
\usepackage{pst-3dplot}

\newcommand{\N}{\mathbb{N}}

\newcommand{\Z}{\mathbb{Z}}

\newcommand{\mB}{\mathcal{B}}
\newcommand{\mC}{\mathcal{C}}
\newcommand{\mH}{\mathcal{H}}

\newcommand{\mG}{\mathcal{G}}

\usepackage{color}

\def\ini{{\rm in}}
\def\m1mn{m_1<\cdots<m_n}
\def\setm1mn{\{m_1,\ldots,m_n\}}
\def\cmreg{Castelnuovo-Mumford regularity }
\def\lp{the degree reverse lexicographic order}
\def\wrtlp{with respect to the degree reverse lexicographic order with $x_1>\cdots>x_{n+1}$}

\theoremstyle{plain}
\newtheorem{theorem}{Theorem}[section]
\newtheorem{lemma}[theorem]{Lemma}
\newtheorem{corollary}[theorem]{Corollary}
\newtheorem{proposition}[theorem]{Proposition}

\theoremstyle{definition}

\newtheorem{example}[theorem]{Example}

\newtheorem{remark}[theorem]{Remark}

\title[Alg. invariants of projective mon. curves assoc. to generalized arithmetic sequences]{Algebraic invariants of projective monomial curves associated to generalized arithmetic sequences}

\author{Isabel Bermejo\, *}

\address{Facultad de Ciencias, Secci\'on de Matem\'aticas, Universidad de La Laguna\\
La Laguna, 38200, Spain\\}

\email{ibermejo@ull.es} 

\author{Eva Garc\'ia-Llorente\, *}

\address{Facultad de Ciencias, Secci\'on de Matem\'aticas, Universidad de La Laguna\\
La Laguna, 38200, Spain}

\email{evgarcia@ull.es}

\author{Ignacio Garc\'ia-Marco\, }

\thanks{\noindent The three authors are partially supported by
Ministerio de Econom\'ia y Competitividad, Spain MTM2013-40775-P}

\address{LIP, ENS Lyon - CNRS -UCBL - INRIA, Universit\'e de Lyon UMR 5668, Lyon, France}
\email{ignacio.garcia-marco@ens-lyon.fr, iggarcia@ull.es}

\begin{document}

\begin{abstract}
Let $K$ be an infinite field and let $\m1mn$ be a generalized
arithmetic sequence of positive integers, i.e., there exist $h, d,
m_1 \in\Z^+$ such that $m_i = h m_1 + (i-1)d$ for all $i \in
\{2,\ldots,n\}$. We consider the projective monomial curve
$\mC\subset \mathbb{P}^{n}_{K}$ parametrically defined by
\begin{center}$x_1=s^{m_1}t^{m_n-m_1},\dots,x_{n-1}=s^{m_{n-1}}t^{m_n-m_{n-1}},x_n=s^{m_n},x_{n+1}=t^{m_n}.$\end{center}
In this work, we characterize the Cohen-Macaulay and Koszul
properties of the homogeneous coordinate ring $K[\mC]$ of $\mC$.
Whenever $K[\mC]$ is Cohen-Macaulay we also obtain a formula for its
Cohen-Macaulay type. Moreover, when $h$ divides $d$, we obtain a
minimal Gr\"obner basis $\mG$ of the vanishing ideal of $\mC$ with
respect to the degree reverse lexicographic order. From $\mG$ we
derive formulas for the Castelnuovo-Mumford regularity, the Hilbert
series and the Hilbert function of $K[\mC]$ in terms of the
sequence.
\end{abstract}

\maketitle

\section{Introduction}

Let $R = K[x_1,\dots,x_{n+1}]$ be the polynomial ring in $n+1$
variables over an infinite field $K$, where $n \geq 2$. Let $\m1mn$
be positive integers and consider the projective monomial curve
$\mC\subset \mathbb{P}^{n}_{K}$ parametrically defined by
$$
x_1=s^{m_1}t^{m_n-m_1},\dots,x_{n-1}=s^{m_{n-1}}t^{m_n-m_{n-1}},x_n=s^{m_n},x_{n+1}=t^{m_n}.
$$
Consider the $K$-algebra homomorphism $\varphi : R \rightarrow
K[s,t]$ induced by $\varphi(x_i)=s^{m_i} t^{m_n-m_i}$ for
$i\in\{1,\dots,n-1\}$, $\varphi(x_n) = s^{m_n}$ and
$\varphi(x_{n+1})= t^{m_n}$. Since $K$ is infinite, the prime ideal
$ker(\varphi)\subset R$ is the vanishing ideal $I(\mC)$ of $\mC$ (see,
e.g., \cite[Corollary 7.1.12]{Vi01}). Thus,  $R / I(\mC)$ is the homogenous coordinate ring
$K[\mC]$ of $\mC$. Moreover, $I(\mC)$ is a
homogeneous binomial ideal (see, e.g., \cite[Lemmas 4.1 and
4.14]{Stu96}).

\medskip The {\it \cmreg}, or {\it
regularity}, of $K[\mC]$ is defined as follows: if
$$0\rightarrow\bigoplus_{j=1}^{\beta_p}R(-e_{pj})\xrightarrow{ }\cdots\xrightarrow{ }
\bigoplus_{j=1}^{\beta_1}R(-e_{1j})\xrightarrow{ } R \xrightarrow{}
K[\mC] \rightarrow 0$$ is a minimal graded free resolution of
$K[\mC]$, then  ${\rm reg}(K[\mC]) := {\rm max}\{e_{ij}-i;\ 1\leq
i\leq p,\ 1\leq j\leq\beta_{i}\}$ (see, e.g.,
 \cite{EiGo84} or \cite{BaMum93} for equivalent
definitions).  Whenever $K[\mC]$ is Cohen-Macaulay, the integer
$\beta_p$ is the {\it Cohen-Macaulay type} of $K[\mC]$; in addition,
if $\beta_p = 1$, the $K$-algebra $K[\mC]$ is
 {\it Gorenstein}. The corresponding ideal, $I(\mC)$, is a {\it
 complete intersection} if it can be generated by a set of $n-1$
polynomials. $K[\mC]$ is said to be a {\it Koszul algebra} if
 its residue class field $K$ has a linear free resolution as $K[\mC]$-module.
The study of Koszul algebras is a topic of considerable current
interest, for a recent account of the theory we refer the reader to
\cite{ConcaDeNegriRossi2013}.

\medskip For all $s \in \N$, we denote by $R_s$ the $K$-vector space
spanned by all homogeneous polynomials of degree $s$.  The {\it
Hilbert function} and {\it Hilbert series} of $K[\mC]$ are
$H\!F_{K[\mC]}(s) := {\rm dim}_K(R_s / I(\mC)\cap R_s)$ for all $s \in
\N$, and $\mH_{K[\mC]}(t) := \sum_{s \in \N} H\!F_{K[\mC]}(s)\,
t^s$, respectively.

\medskip
In this paper we study the already mentioned invariants and
properties when $\m1mn$ is a generalized arithmetic sequence of
relatively prime integers, i.e., there exist $h, d, m_1 \in\Z^+$
such that $m_i = hm_1 + (i-1)d$ for each $i\in\{2,\ldots,n\}$ and
$\gcd\{m_1,d\}=1$. In this context, we prove that $K[\mC]$ is
Cohen-Macaulay if and only if $\m1mn$ is an arithmetic sequence. It
is remarkable that this result is obtained as a consequence of a new
result providing a large family of projective monomial curves that
are not arithmetically Cohen-Macaulay. We also provide an
alternative proof of \cite[Theorem 6.1]{BerGarMarCons}, where we
determine that $I(\mC)$ is a complete intersection if and only if
either $n = 2$ or $n = 3$, $h = 1$ and $m_1$ is even. Moreover, we
prove that $K[\mC]$ is Koszul if and only if $m_1,\ldots,m_n$ are
consecutive numbers and $n > m_1$.

\medskip In addition, we study in detail when $\m1mn$ is a generalized arithmetic sequence and
$h$ divides $d$. In this setting we provide a formula for ${\rm
reg}(K[\mC])$. To this end, our main reference is \cite{BerGi06}. In
this paper, the authors prove that for any projective monomial curve
$\mC$, the regularity of $K[\mC]$ is equal to the regularity of $R /
\ini(I(\mC))$, where ${\rm in}(I(\mC))$ denotes the initial ideal of
$I(\mC)$ with respect to the degree reverse lexicographic order with
$x_1>\cdots>x_{n+1}$ (see \cite[Section 4]{BerGi06}). Moreover, they
obtain that the regularity of $R / \ini(I(\mC))$ is the maximum of
the regularities of the irreducible components of $\ini(I(\mC))$
(see \cite[Corollary 3.17]{BerGi06}). Since the regularity of an
irreducible monomial ideal is easy to compute, our strategy consists
of obtaining a minimal Gr\"obner basis $\mG$ of $I(\mC)$ with
respect to the degree reverse lexi\-co\-graphic order with $x_1 > \cdots
> x_{n+1}$. From $\mathcal G$, we get the irreducible components of
$\ini(I(\mC))$, which allow us to deduce the value of ${\rm
reg}(K[\mC])$ in terms of the arithmetic sequence $\m1mn$. For
obtaining this result we separate two cases: the Cohen-Macaulay
case, i.e., when $\m1mn$ is an arithmetic sequence (Section 2); and
the non Cohen-Macaulay one, i.e., when $h\geq 2$ and $h$ divides $d$
(Section 3).

\medskip We also exploit the knowledge of $\mG$ to
describe Hilbert series and the Hilbert function of $K[\mC]$.
Additionally, in the Cohen-Macaulay case, we provide a formula for
its Cohen-Macaulay type yielding a characterization of the
Gorenstein property. The results in the Cohen-Macaulay case extend
in a natural way those of \cite{MolinelliTamone1995} that consider
the particular setting where $\{m_1,\ldots,m_n\}$ form a minimal set
of generators of the semigroup they generate.


\medskip The contents of the paper are the following. Section 2 is devoted to arithmetic sequences.
The key result is Theorem \ref{gbtheorem}, which provides $\mG$.
As a byproduct, we derive in Corollary \ref{CcohenMacaulay} that the
homogeneous coordinate ring of any projective monomial curve
associated to an arithmetic sequence is Cohen-Macaulay. Proposition \ref{irredDec} gives the irreducible
 components of $\ini(I(\mC))$. As a consequence of this, in
Theorem \ref{regfor}, we get the value of the regularity of
$K[\mC]$ in terms of the sequence $\m1mn$. In Theorems
\ref{hilbseries}, \ref{hilbfunction} and \ref{CMtype} we provide
 the Hilbert series, the Hilbert function and the
Cohen-Macaulay type of $K[\mC]$. In addition, we give
in Corollary \ref{gorenstein} a characterization of the Gorenstein
property for $K[\mC]$.

\medskip
In the third section we begin by providing in Theorem
\ref{notCM_GenArith} a large family of projective monomial curves
whose corresponding coordinate ring is not Cohen-Macaulay. More
precisely, we prove that $K[\mC]$ is not Cohen-Macaulay
whenever $\mC$ is the projective monomial curve associated to a
sequence of integers containing a generalized arithmetic sequence of
at least $l\geq 3$ terms satisfying some additional hypotheses. The rest
of this section concerns when $\m1mn$ is a gene\-ra\-lized arithmetic
sequence and $n \geq 3$. In particular, as a byproduct of Theorem
\ref{notCM_GenArith}, we derive in Corollary
\ref{CMcharacteriz_GenArith} that in this setting $K[\mC]$ is not
Cohen-Macaulay unless $\m1mn$ is an arithmetic sequence, i.e., if $h
= 1$. As a consequence of this we characterize in Corollary \ref{IC}
the complete intersection property for $I(\mC)$. Then, we study in
detail the setting where $h > 1$ and $h$ divides $d$. Under these
hypotheses, we obtain the Gr\"obner basis $\mG$ in Theorem
\ref{GBgeneraliz_h_div_d}, which also is a minimal set of generators
of $I(\mC)$. In Proposition \ref{irredDecGeneralized} we present the
irreducible decomposition of $\ini(I(\mC))$. The formula for ${\rm
reg}(K[\mC])$ is provided in Theorem \ref{reg_generalizec_h_div_d},
while in Corollary \ref{ultimopaso} we prove that the regularity is
always attained at the last step of the minimal graded free
resolution. The Hilbert series, the Hilbert function and the Hilbert
polynomial are given in Theorems \ref{HilbertSeriesGeneralized},
\ref{HilbFunctionGeneralized} and Corollary
\ref{HilbertPolynomial_generalized}, respectively.

\medskip The last section is devoted to the study of the
Koszul property for $K[\mC]$. In particular,
 Theorem \ref{koszulGeneralizedArithmetic} characterizes when $K[\mC]$ is a Koszul algebra
when $\m1mn$ is a generalized arithmetic sequence. Finally, we also
characterize the Koszul property when $n=3$ (Theorem
\ref{Koszulcodim2}) and when $n=4$ (Theorem \ref{Koszulcodim3}).

\section{Arithmetic sequences}

This section concerns the study of $K[\mC]$ when $\m1mn$ an
arithmetic sequence of relatively prime integers and $n \geq 2$. The
key point of this section is to prove Theorem \ref{gbtheorem}, where
we describe a minimal Gr\"obner basis $\mG$ of $I(\mC)$ \wrtlp.

\medskip We begin with a technical lemma. In it we associate
to the sequence $\m1mn$ a pair $(\alpha, k) \in \N^2$. This pair of
integers plays an important role in the rest of the paper.

%

\begin{lemma}\label{uniquenessLema}
Take $q \in \N$ and $r \in \{1,\ldots,n-1\}$ such that $m_1 = q
(n-1) + r$ and set $\alpha := q + d$ and $k := n - r$. Then,
\begin{enumerate}
\item[{\rm (a)}] $\alpha m_1 + m_i = m_{n-k+i} + q m_n$ for all $i \in \{1,\ldots,k\}$,


\item[{\rm (b)}] $\alpha + 1 = {\rm min}\,\{b\in\Z^+\ \vert\
b\,m_1\in\sum_{i=2}^{n} \N m_i\}$.
\end{enumerate}
\end{lemma}

\begin{proof} Being $(a)$ trivial, let us prove $(b)$.
Set $M := {\rm min}\,\{b\in\Z^+\ \vert\ b\,m_1\in\sum_{i=2}^{n} \N
m_i\}$. By $(a)$ it follows that $(\alpha + 1) m_1 = m_{n-k+1} + q
m_n,$ so $M \leq \alpha+1$. Assume by contradiction that $M \leq
\alpha$. We have that
\begin{equation} \label{exp2} M m_1 = \sum_{i = 2}^{n} \delta_i
m_i {\rm \ for\ some\ \delta_2,\ldots,\delta_{n} \in
\N}.\end{equation} We are proving the following facts which clearly
lead to a contradiction:
\begin{itemize} \item[(i)] $ M - \sum_{i = 2}^{n} \delta_i > 0,$
\item[(ii)] $M - \sum_{i = 2}^{n} \delta_i < d$, and \item[(iii)] $d$ divides
$M - \sum_{i=2}^{n} \delta_i$. \end{itemize} From (\ref{exp2}) we
get (i). Also from (\ref{exp2}) we have $(M - \sum_{i = 2}^{n}
\delta_i)\,  m_1 = \sum_{i = 2}^{n} (i - 1) \delta_i d$, which gives
(iii) because $\gcd\{m_1,d\} = 1$. From (\ref{exp2}) and the
assumption $M \leq \alpha$, we obtain $(\sum_{i = 2}^{n} \delta_i)\,
m_n \geq \sum_{i = 2}^n \delta_i m_i = M m_1 = M m_n - M (n-1) d
\geq M m_n - \alpha (n-1) d > (M - d)\, m_n$. This proves (ii) and
completes the proof of $(b)$.

 \end{proof}

\medskip Following the notation of Lemma \ref{uniquenessLema}, we can
state the main result of this section.

\begin{theorem}\label{gbtheorem}
$$\mG = \{x_ix_j-x_{i-1}x_{j+1}\ \vert\ 2\leq i\leq j \leq n-1\}\,\cup \{x_1^{\alpha}x_{i}-x_{n+1-i}x_n^{q}x_{n+1}^{d}\ \vert\ 1\leq i\leq k\}$$
is a minimal Gr\"obner basis of $I(\mC)$ with respect to the degree reverse
lexicographic order.
\end{theorem}

\begin{proof} One can easily check that
$\mG\subseteq I(\mC)$ by using Lemma \ref{uniquenessLema}.(a). Let us
prove that $\ini(I(\mC))\subseteq\langle{\rm in}(\mG)\rangle$ where $\ini
(\mG)=\{x_i x_j\ \vert\ 2\leq i\leq j\leq n-1\} \cup \{x_1^{\alpha}
x_i\ \vert\ 1\leq i\leq k\}$. By contradiction, suppose that $\ini(I(\mC))
\not \subseteq \langle{\rm in}(\mG)\rangle$. We claim that there
exist two monomials $l_1,l_2 \notin \langle{\rm in}(\mG)\rangle$
such that $l_1 - l_2 \in I(\mC)$. Indeed, if $\ini(I(\mC)) \not \subseteq
\langle{\rm in}(\mG)\rangle$, then the cosets of $\{l \, \vert \, l$
is a monomial and $l \notin \langle{\rm in}(\mG)\rangle\}$ do not
form a $K$-basis of $K[\mC]$. Then there exist monomials
$l_1,\ldots,l_t \notin \langle{\rm in}(\mG)\rangle$ and
$\lambda_1,\ldots, \lambda_t \in K \setminus \{0\}$ such that $f =
\sum_{i = 1}^t \lambda_i l_i \in I(\mC)$. Since $\varphi(f) = 0$ and
$\varphi(l_1) \neq 0$, there exists $j \in \{2,\ldots,t\}$ such that
$\varphi(l_1) = \varphi(l_j)$. Suppose that $j = 2$, then $l_1 - l_2
\in I(\mC)$ and $l_1,l_2 \notin \langle{\rm in}(\mG)\rangle$. We can
also assume that $\gcd \{l_1,l_2\} = 1$ because $I(\mC)$ is prime.

Since ${\rm in}(\mG) = \{x_i x_j\ \vert\ 2\leq i\leq j\leq n-1\}
\cup \{x_1^{\alpha} x_i\ \vert\ 1\leq i\leq k\}$, we may assume that
$l_1=x_1^a x_i^{\delta_i} x_n^{b_1} x_{n+1}^{c_1}$ and $l_2 =
x_j^{\delta_j} x_n^{b_2} x_{n+1}^{c_2}$, where:
\begin{itemize}
\item[$({\rm a})$] $0 \leq a \leq \alpha$,
\item[$({\rm b})$] $2 \leq i,j \leq n-1$, $i\neq j$,
\item[$({\rm c})$] $\delta_i, \delta_j \in\{0,1\}$,
\item[$({\rm d})$] $b_1, b_2, c_1, c_2 \geq 0$, and
\item[$({\rm e})$] if $a = \alpha$ and $\delta_i = 1$, then $i > k$ (otherwise
$l_1 \in \langle{\rm in}(\mG)\rangle$).
\end{itemize}

Moreover, since $\varphi(l_1)=\varphi(l_2)$, we get that $a
m_1+\delta_i m_i+b_1 m_n = \delta_j m_j + b_2 m_n$.

If $a = 0$, we may  assume that $b_1 = 0$ because $\gcd\{l_1,l_2\} =
1$. However, in this setting we have that $\delta_i m_i = \delta_j
m_j + b_2 m_n$, but this implies that $b_2 = 0$ and that $m_i =
m_j$. Hence $l_1 = l_2$, a contradiction.

From now on, we assume that $a \geq 1$. Firstly, we observe that
$b_1=0$; otherwise $b_2=0$ and $a m_1+ \delta_i m_i+b_1 m_n\geq
m_n>\delta_j m_j$, a contradiction. Hence, we get that
\begin{equation}\label{BGigualdad_1}a m_1+\delta_i m_i = \delta_j m_j +
b_2m_n.\end{equation}

We consider several different cases:

\emph{Case 1:} If $\delta_i = 0$ then (\ref{BGigualdad_1}) becomes
$a m_1 = \delta_j m_j + b_2 m_n$. By Lemma
\ref{uniquenessLema}.(b), we get that $a \geq \alpha + 1$, a
contradiction.

 \emph{Case 2:} If $\delta_j = 0$ then
(\ref{BGigualdad_1}) becomes $a m_1 + \delta_i m_i = b_2 m_n$. From
Lemma \ref{uniquenessLema} we get that $a = \alpha$, then $\delta_i
= 1$ and $i = k$, a contradiction.

\emph{Case 3:} If $\delta_i=\delta_j=1$, then $am_1+m_i=m_j+b_2 m_n$
from (\ref{BGigualdad_1}).

\emph{Subcase 3.1:} If $i > j$, one can derive that $a m_1  =
m_{n-i+j} + (b_2-1) m_n$. By Lemma \ref{uniquenessLema}.(b), this
implies that $a \geq \alpha + 1$, a contradiction.

\emph{Subcase 3.2:} If $i < j$, one can easily obtain that  $a m_1 +
m_{n-j+i} = (b_2 + 1) m_n,$  which implies that $a \geq \alpha$.
Thus, $a = \alpha$ and $n-j+1=k$, a contradiction to $(e)$.

Moreover, one can easily check that $\mG$ is minimal.


\end{proof}

We observe in  Theorem \ref{gbtheorem} that the variables $x_{n}$
and $x_{n+1}$ do not appear in the minimal set of generators of
$\ini(I(\mC))$. Thus we apply \cite[Proposition 2.1]{BerGi01} to
deduce the following result.

\begin{corollary}\label{CcohenMacaulay}
$K[\mC]$ is Cohen-Macaulay.
\end{corollary}

It is worth mentioning that this result is a generalization of
\cite[Proposition 1.2]{MolinelliTamone1995} where the authors obtain
the result in a more restrictive setting, i.e., when
$\{m_1,\ldots,m_n\}$ is a minimal system of generators of the
semigroup $\sum_{i=1}^{n}\mathbb{N} m_i$. It is worth mentioning
that our techniques are different from theirs.

\begin{remark} \label{sistemaMinimal} The Gr\"obner basis $\mG$ provided in Theorem
\ref{gbtheorem} is also a minimal set of generators of $I(\mC)$, see
\cite[Remark 3.13]{LiPatilRoberts12}. Thus, the first Betti number
$\beta_1$ in the minimal graded free resolution of $K[\mC]$ equals
$\binom{n-1}{2}+k$.
\end{remark}

Our next objective is to provide an explicit description of the
irredundant irre\-du\-cible decomposition of the initial ideal
$\ini(I(\mC))$ of $I(\mC)$ with respect to the degree reverse
lexicographic order. An ideal is {\it irreducible} if it cannot be
written as intersection of two ideals properly containing it. A
monomial ideal is irreducible if it is generated by pure powers of
variables (see, e.g., \cite[Proposition 5.1.6]{Vi01}). Moreover, an
{\it irreducible decomposition} of a monomial ideal is an expression
of it as intersection of irreducible monomial ideals. If there is no
redundant ideal in the intersection, this expression is unique (see,
e.g., \cite[Theorem 5.1.17]{Vi01}) and we call it {\it the
irredundant irreducible decomposition} of $I$.

\begin{proposition}\label{irredDec}
Let $\alpha,k$ be the as in Lemma \ref{uniquenessLema}.
\begin{enumerate}
\item If $k<n-1$, then
\begin{center} $\ini(I(\mC)) = \bigcap_{i=2}^{n-1}\,\langle
x_1^{\alpha+\delta_i},x_2,\ldots,x_{i-1},x_i^2,x_{i+1},\ldots,x_{n-1}\rangle$\end{center}
is the irredundant irreducible decomposition of
$\ini(I(\mC))$, where $\delta_i = 0$ if $i \in \{2,\ldots,k\}$ and
$\delta_i = 1$ if $i \in \{k+1,\ldots,n-1\}$.

\smallskip

\item If $k=n-1$, then
\begin{center}$\begin{array}{ccl}
  \ini(I(\mC)) & = & \left[\bigcap_{i=2}^{n-1}\,\langle
x_1^{\alpha},x_2,\ldots,x_{i-1},x_i^2,x_{i+1},\ldots,x_{n-1}\rangle\right] \\
   & & \bigcap\langle x_1^{\alpha+1},x_2,\ldots,x_{n-1}\rangle
\end{array}$ \end{center}
is the irredundant irreducible decomposition of $\ini(I(\mC))$.
\end{enumerate}
\end{proposition}
\begin{proof} Let us denote $A_i:=\langle
x_1^{\alpha+\delta_i},x_2,\ldots,x_{i-1},x_i^2,x_{i+1},\ldots,x_{n-1}\rangle$,
for every $i\in\{2,\ldots,n-1\}$, with $\delta_i = 0$ if $i \in
\{2,\ldots,k\}$ and $\delta_i = 1$ if $i \in \{k+1,\ldots,n-1\}$,
and $B:=\langle x_1^{\alpha+1},x_2,\ldots,x_{n-1}\rangle$. We denote
$A:=\bigcap_{i=2}^{n-1}A_i$ if $k<n-1$, or
$A:=\bigcap_{i=2}^{n-1}A_i\cap B$ otherwise, and we aim at proving
that $\ini(I(\mC)) = A$. From Theorem \ref{gbtheorem}, we have that
$\{x_ix_j\ |\ 2\leq i\leq j\leq n-1\}\cup\{x_1^{\alpha}x_i\ |\ 1\leq
i\leq k\}$ is a minimal system of generators of $\ini(I(\mC))$
\wrtlp. We observe that $\ini(I(\mC))\subseteq A_i$ for all
$i\in\{2,\ldots,n-1\}$ and, if $k=n-1$, that $\ini(I(\mC))\subseteq
B$. Thus $\ini(I(\mC))\subseteq A$.

In order to see the reverse inclusion, we take a monomial
$l\notin\ini(I(\mC))$. We have three possibilities:
\begin{itemize}
\item[{\rm (a)}] $l=x_1^ax_i x_n^b x_{n+1}^c$ with $0\leq a <\alpha$,
$i\in\{2,\ldots,n-1\}$ and $b,c\geq 0$,

\item[{\rm (b)}] $l=x_1^a x_n^b x_{n+1}^c$ with $0\leq a\leq\alpha$ and $b, c\geq
0$, or

\item[{\rm (c)}] $l=x_1^{\alpha}x_i x_n^b x_{n+1}^c$ with $i\in\{k+1,\ldots,n-1\}$ and $b, c\geq 0$.
\end{itemize}
If (a) holds, then $l\notin A_i$. In (b), we have that $l\notin
A_{n-1}$ if $k<n-1$ and $l\notin B$ if $k=n-1$.  In (c), we observe
that $k< i \leq n-1$, and then $l\notin A_i$. Thus,
$\ini(I(\mC))=A$.

It only remains to see that there is no redundant ideal in the
intersection. For all $i\in\{2,\ldots,n-1\}$, we observe that
$x_i\in A_j$ for all $j\neq i$, and $x_i\in B$ if $k=n-1$, but
$x_i\notin\ini(I(\mC))$; hence $A_i$ is not redundant. If $k=n-1$, then
$x_1^{\alpha}\in A_i$ for all $i\in\{2,\ldots,n-1\}$, but
$x_1^{\alpha}\notin\ini(I(\mC))$, hence $B$ is not redundant.

\end{proof}

As we claimed in the Introduction, ${\rm reg}(K[\mC]) = {\rm reg}(R
/ \ini(I(\mC)))$ and the regu\-la\-ri\-ty of $R / \ini(I(\mC))$ can be expressed in terms
of the Castelnuovo-Mumford regularity of the irreducible components
of $\ini(I(\mC))$. This is due to the fact that $\ini(I(\mC))$ is a monomial ideal
of {\it nested type}, i.e., its associated primes are of the form
$(x_1,\ldots,x_i)$ for various $i$. The following result states this
relation.

\begin{lemma}\cite[Corollary 3.17]{BerGi06}\label{regNT}
Let $I\subset R$ be a monomial ideal of nested type, and let
$I=\mathfrak{q}_1\cap\cdots\cap\mathfrak{q}_r$ be its irredundant
irreducible decomposition. Then,
$${\rm reg}(R/I)={\rm max}\{{\rm reg}(R/\mathfrak{q}_i);\ 1\leq i\leq r\}$$
\end{lemma}

%

Moreover, for $\mathfrak q$ an irreducible monomial ideal whose
monomial minimal generators have degrees $d_1,\ldots,d_s$ it is
known that ${\rm reg}(R / \mathfrak q) =  d_1+\cdots+d_s-s$ (see,
e.g., \cite[Remark 2.3]{BerGi01}).

\smallskip

The following result gives the value of the regularity of $K[\mC]$
in terms of the arithmetic sequence.

\begin{theorem}\label{regfor}
Let $m_1<\cdots<m_n$ be an arithmetic sequence with
$\gcd\,\{m_1,d\}=1$. Then, $${\rm reg}(K[\mC])= \left\lceil
\frac{m_n - 1}{n - 1} \right\rceil.$$
\end{theorem}
\begin{proof} We take $\alpha,k$ as in Lemma
\ref{uniquenessLema}. A direct application of Proposition \ref{irredDec} and
 Lemma \ref{regNT} yields
$${\rm reg}(K[\mC])=\left\{ \begin{array}{ll} \alpha + 1 & {\rm if }\ k < n-1, \\
\alpha & {\rm if}\ k = n-1.
\end{array} \right.$$

We observe that $k = n-1$ if and only if $m_n \equiv 1 \ ({\rm mod}
\ n-1)$. Hence, if $k = n -1$, then $\alpha = q + d = \left\lfloor
\frac{m_1-1}{n-1}\right\rfloor + d = \left\lfloor
\frac{m_n-1}{n-1}\right\rfloor
 = \left\lceil
\frac{m_n-1}{n-1}\right\rceil.$ Moreover, if $k \neq n-1$, then
$\alpha + 1 = q + d + 1 = \left\lfloor
\frac{m_1-1}{n-1}\right\rfloor + d + 1 = \left\lfloor
\frac{m_n-1}{n-1}\right\rfloor + 1   = \left\lceil
\frac{m_n-1}{n-1}\right\rceil$.
\end{proof}

Now we exploit the description of $\ini(I(\mC))$ to provide an
explicit description of the Hilbert series, the Hilbert function and
the Cohen-Macaulay type of $K[\mC]$. These results have also been
obtained in \cite{MolinelliTamone1995} in the more restrictive
setting when $\{m_1,\ldots,m_n\}$ is a minimal system of generators
of  $\sum_{i=1}^{n}\N m_i$. Before going to the particular case when
$\m1mn$ is an arithmetic sequence, we introduce a general result
expressing the Hilbert series of $K[\mC]$ in terms of the monomials
not belonging to $\ini(I(\mC))$. This result applies to any
projective monomial curve.

\begin{proposition}\label{hilbseriesgeneral} Let $\m1mn$ be a sequence of positive integers.
The Hilbert series of $K[\mC]$ is
$$ \mH_{K[\mC]}(t) = \frac{\sum_{x^{\gamma} \notin J + \langle x_n, x_{n+1}\rangle} t^{\, \vert \gamma \vert} -
t \sum_{x^{\gamma} \notin J + \langle x_{n+1} \rangle \atop x_n
x^{\gamma} \in J} t^{\, \vert \gamma \vert}}{(1-t)^2},$$ where
$\vert \gamma \vert := \sum_{i = 1}^{n+1} \gamma_i$ for $\gamma =
(\gamma_1,\ldots,\gamma_{n+1}) \in \N^{n+1}$ and $J :=
\ini(I(\mC))$. Moreover, if $K[\mC]$ is Cohen-Macaulay, then
$$\mH_{K[\mC]}(t)= \frac{\sum_{x^{\gamma} \notin J + \langle x_n, x_{n+1} \rangle} t^{\, \vert \gamma \vert} }
{(1-t)^2}.$$
\end{proposition}

\begin{proof}
 It is well known that the Hilbert series of $K[\mC]$
 coincides with that of $R / J$,
then $\mH_{K[\mC]}(t) = \sum_{x^{\gamma} \notin J} t^{\, \vert
\gamma \vert}$. We observe that
$$\mH_{K[\mC]}(t)=\sum_{x^{\gamma} \notin J \atop \gamma_{n+1} = 0}
t^{\, \vert \gamma \vert} + \sum_{x^{\gamma} \notin J \atop
\gamma_{n+1}\neq 0} t^{\, \vert \gamma \vert}.$$ Since we are
considering the degree reverse lexicographic order and $I(\mC)$ is a
homogeneous prime ideal, then $x_{n+1}$ does not appear in any
minimal generator of $J$. Thus, it follows that $\{x^{\gamma} \notin
J \, \vert \, \gamma_{n+1} \neq 0\} = \{x_{n+1} x^{\beta} \, \vert
\, x^{\beta} \notin J \}.$ Hence,
$$\mH_{K[\mC]}(t) = \sum_{x^{\gamma}\notin J \atop \gamma_{n+1} = 0} t^{\, \vert \gamma \vert}
+ t \, \sum_{x^{\beta}\notin J} t^{\, \vert \beta \vert} =
\sum_{x^{\gamma}\notin J \atop \gamma_{n+1} = 0} t^{\, \vert \gamma
\vert} + t \, \mH_{K[\mC]}(t).$$ Then, we get that $$(1-t) \,
\mH_{K[\mC]}(t) = \sum_{x^{\gamma}\notin J \atop \gamma_{n+1} = 0}
t^{\,\vert \gamma \vert} = \sum_{x^{\gamma}\notin J \atop \gamma_n =
\gamma_{n+1} = 0} t^{\, \vert \gamma \vert} + \sum_{x^{\gamma}\notin
J \atop \gamma_n \neq 0,  \ \gamma_{n+1} = 0} t^{\, \vert \gamma
\vert}.$$ Moreover,
$$\sum_{x^{\gamma} \notin J
\atop \gamma_{n} \neq 0,\,\gamma_{n+1} = 0} t^{\,\vert \gamma \vert}
= t \sum_{x^{\beta} \notin J \atop \beta_{n+1}=0} t^{\,\vert \beta
\vert} - t \sum_{x^{\beta} \notin J,\, x_nx^{\beta}\in J \atop
\,\beta_{n+1}=0} t^{\,\vert \beta \vert} .$$ To conclude the result
it suffices to observe that
$$t \sum_{x^{\beta} \notin J \atop
\beta_{n+1}=0} t^{\,\vert \beta \vert} = t (1-t)\, \mH_{K[\mC]}(t).
$$ Additionally, when $K[\mC]$ is
Cohen-Macaulay, by \cite[Proposition 2.1]{BerGi01} we have that $x_n
x^{\gamma} \in J$ if and only if $x^{\gamma} \in J$ and the result
follows.
\end{proof}

As a consequence, when $\m1mn$ is an arithmetic sequence of
relatively prime integers, we obtain the following result.

\begin{theorem}\label{hilbseries} The Hilbert series of $K[\mC]$ is
$$ \mH_{K[\mC]}(t)= \frac{1+(n-1)(t+\cdots+t^{\alpha})+(n-1-k)t^{\alpha+1}}{(1-t)^2}. $$
\end{theorem}
\begin{proof} Since $K[\mC]$ is Cohen-Macaulay, by Corollary
\ref{CcohenMacaulay}, we have that \begin{center}$(1-t)^2 \, \mH_{K[\mC]}(t) =
\sum_{x^{\gamma}\notin\ini(I(\mC)) + \langle x_n, x_{n+1}\rangle}
t^{\, \vert \gamma \vert}.$\end{center} Now, we apply Theorem \ref{gbtheorem} to
get that $\{x^{\gamma} \, \vert \, x^{\gamma} \notin\ini(I(\mC)) +
\langle x_n, x_{n+1} \rangle \} = \{1\}\cup \{x_1^ax_i\ \vert\ 1\leq
i\leq n-1\ {\rm and}\ 0\leq a\leq \alpha-1\}\cup \{x_1^{\alpha}x_i\
\vert\ k+1\leq i\leq n-1\}$ and we deduce the result.

\end{proof}

\begin{remark}It is worth pointing out that from Theorem
\ref{hilbseries} we can also get the value of {\rm reg}$(K[\mC])$
because $K[\mC]$ is Cohen-Macaulay. Indeed, the regularity of
$K[\mC]$ is attained at the last step of a minimal graded free
resolution (see, e.g., \cite[Proposition 1.1]{BerGi00}). Moreover,
the Hilbert series of $K[\mC]$ is related to a minimal graded free
resolution by the formula $\mH_{K[\mC]}(t) = \frac{1}{(1-t)^{n+1}}(1
+ \sum_{i=1}^{p}(-1)^{i}\sum_{j=1}^{\beta_{i}} t^{e_{ij}})$. Thus,
${\rm reg}(K[\mC])$ equals the degree of the $h$-polinomial. Since
the $h$-polynomial of $K[\mC]$ is
$1+(n-1)(t+\cdots+t^{\alpha})+(n-1-k)t^{\alpha+1}$, then ${\rm
reg}(K[\mC]) = \alpha$ if $k = n-1$ and ${\rm reg}(K[\mC]) = \alpha
+ 1$ if $k < n-1$.
\end{remark}

\begin{theorem}\label{hilbfunction}
The Hilbert function of $K[\mC]$ is
$$H\!F_{K[\mC]}(s)= \left\{
\begin{array}{ll} \binom{s+2}{2} + (n-2)\binom{s+1}{2} &\ \  {\rm if}\ 0 \leq s < \alpha, \\
m_n\, s + \alpha(2 - n + k) - \binom{\alpha + 1}{2} - (n-2)
\binom{\alpha}{2} + 1 & \ \ {\rm if}\  s \geq \alpha.
\end{array} \right.$$
\end{theorem}
\begin{proof} The Hilbert function of $K[\mC]$ coincides with that of $R / \ini(I(\mC))$.
Thus from Theorem \ref{gbtheorem}, we have that $$
\begin{array}{lll} K[\mC] \simeq \bigoplus_{\gamma, \beta \in \N} & &
\left\{ \left(\bigoplus_{0 \leq \lambda \leq \alpha}
K\{x_1^{\lambda} x_n^{\beta} x_{n+1}^{\gamma}\}\right)
 \bigoplus   \left(\bigoplus_{0 \leq \lambda < \alpha \atop 2 \leq i
\leq n-1} K\{x_1^{\lambda}x_i x_n^{\beta}x_{n+1}^{\gamma}\}\right)
 \right. \\ & & \left. \bigoplus \ \left(\bigoplus_{k+1 \leq i \leq
n-1} K\{x_1^{\alpha}x_i
 x_n^{\beta}x_{n+1}^{\gamma}\}\right) \right\} \end{array}
 $$
 Then, for $0 \leq s\leq \alpha$, we have that
 $$H\!F_{K[\mC]}(s) = \binom{s+2}{2} + (n-2)\binom{s+1}{2}.$$

\noindent And for $s\geq \alpha +1$, we have that
$$\begin{array}{rl} H\!F_{K[\mC]}(s) & = \sum_{i=0}^{\alpha}(s+1-i)+(n-2)\sum_{i=0}^{\alpha-1}(s-i) +(n-1-k)(s-\alpha) \\
& = (s+1)(\alpha+1) - \binom{\alpha+1}{2} + (n-2)\left( s\alpha -
\binom{\alpha}{2}\right)+(n-1-k)(s-\alpha) \\ &= \left((n-1)\alpha
+n-k\right)s + \alpha(2 - n + k) - \binom{\alpha + 1}{2} - (n-2)
\binom{\alpha}{2} + 1.\end{array}
$$ Finally, we observe from Lemma \ref{uniquenessLema} that $\alpha = q+d$
and then $(n-1)\alpha +n-k = (n-1)(q+d)+n-k = m_n$.

\end{proof}

\begin{remark}\label{regularityHilbFunction}
Since $K[\mC]$ is Cohen-Macaulay, the Castelnuovo-Mumford
re\-gu\-la\-rity of $K[\mC]$ is one unit bigger than the regularity
of its Hilbert function (see, e.g., \cite[Theorem 2.5]{BerGi01}).
The {\it regularity of the Hilbert function} of $K[\mC]$ is the
smallest integer $s_0$ such that for all $s \geq s_0$ its Hilbert
polynomial and its Hilbert function in $s$ coincide. Thus, using
this approach, Theorem \ref{hilbfunction} allows us to reprove
Theorem \ref{regfor}. Indeed,
 the Hilbert polynomial of $K[\mC]$ is $P_{K[\mC]}(s)
:= m_n\, s + \alpha(2 - n + k) - \binom{\alpha + 1}{2} - (n-2)
\binom{\alpha}{2} + 1$ and $P_{K[\mC]}(s) = H\!F_{K[\mC]}(s)$ for
all $s \geq \alpha$. Moreover, it is easy to check that
$P_{K[\mC]}(\alpha - 1) \neq H\!F_{K[\mC]}(\alpha - 1)$ if and only
if $k < n-1$. Nevertheless, if $k = n-1$, then $P_{K[\mC]}(\alpha -
2) \neq H\!F_{K[\mC]}(\alpha - 2)$. Thus, the regularity of the
Hilbert function equals $\alpha$ if $k < n-1$, or $\alpha - 1$ if $k
= n-1$.
\end{remark}

Now, we obtain the Cohen-Macaulay type of $K[\mC]$.

\begin{theorem}\label{CMtype}
Take $c \in \N$ and $\tau \in \{1,\ldots,n-1\}$ such that $m_1 - 1 =
c (n-1) + \tau$. Then, the Cohen-Macaulay type of $K[\mC]$ is
$\tau$.
\end{theorem}
\begin{proof} We set $a_i := (m_i, m_n - m_i) \in \N^2$ for all $i \in \{1,\ldots,n\}$
and $a_{n+1} := (0, m_n) \in \N^2$ and consider the grading in $R$
induced by ${\rm deg}(x_i) := a_i$ for all $i \in \{1,\ldots,n\}$
and ${\rm deg}(x_{n+1}) := a_{n+1}$. We also consider the sets $\mB
:= \{{\rm deg}(x^{\gamma}) \, \vert \, x^{\gamma} \notin \ini(I(\mC)) +
(x_n,x_{n+1})\} \subset \N^2$ and $\tilde{\mB} = \{s \in \mB \, \vert
\, s + a_i \notin \mB$ for all $i: 1 \leq i \leq n-1\} \subset \mB$.
By \cite[Theorem (4.9)]{CavaNiesi83}, the Cohen-Macaulay type of
$K[\mC]$ is the number of elements in $\tilde{\mB}$.

By Theorem \ref{gbtheorem}, we have that $$\mB = \{(0,0)\} \cup \{ b
a_1 + a_i \, \vert \, 0 \leq b \leq \alpha-1, 1 \leq i \leq n-1\}
\cup \{ \alpha a_1 + a_i \, \vert \,  k + 1 \leq i \leq n-1\},$$
where $k \in \{1,\ldots,n-1\}$ satisfies that $m_1 + k \equiv 1\
({\rm mod}\ n-1)$. We claim that
 $$\tilde{\mB} = \left\{
\begin{array}{cl} \{  \alpha a_1 + a_i \, \vert \, k + 1
\leq i \leq n-1\} & $ if $k < n-1 \\ \{(\alpha - 1) a_1 + a_i \,
\vert \, 1 \leq i \leq n-1\} & $ if $k = n-1 .
\end{array} \right.$$  For $k < n-1$, if $c = b a_1 + a_i$ with $0 \leq b
\leq \alpha-1$ and $1 \leq i \leq n-1$, then $c + a_{n-i} = b a_1 +
a_i + a_{n-i} = (b+1)a_1 + a_{n-1} \in \mB$ and, hence, $c \notin
\tilde{\mB}$. Moreover, it is easy to check that $\alpha a_1 + a_i \in
\tilde{\mB}$ for all $i: k +1 \leq i \leq n-1.$ For $k = n-1$ the
claim follows easily.

Therefore, the Cohen-Macaulay type of $K[\mC]$ is $n - 1$ if $k =
n-1$, or $n - 1 - k$ otherwise. Thus, the result follows from the
definition of $k$.

\end{proof}

From Theorem \ref{CMtype} we easily deduce the following
characterization of the Gorenstein property for $K[\mC]$.

\begin{corollary}\label{gorenstein}
$K[\mC]$ is Gorenstein $\Longleftrightarrow$ $m_1 \equiv 2\ ({\rm
mod}\ n-1).$
\end{corollary}

\medskip Let us illustrate the results of this section with some examples.

\begin{example}\label{ejemlo1Aritmeticas}
Consider the set $\{m_1,m_2,m_3,m_4,m_5\}$ with $m_1 := 10,\, m_2 :=
13,\, m_3:=16,\, m_4 := 19$ and $m_5 := 22$. We observe that $m_1 =
2 \cdot 4 + 2$, then $q = 2$ and $r = 2$ and we have that $\alpha =
q + d = 5$ and that $k = 5 - 2 = 3$. From Theorem \ref{regfor} we
get that
$${\rm reg}(K[\mC]) = \lceil (22 - 1) / (5 - 1) \rceil = 6.$$ Since, $m_1 - 1 = 2 (n-1) + 1$,
from Theorem \ref{CMtype} we obtain that the  Cohen-Macaulay type is
 $\tau = 1$ and hence $K[\mC]$ is Gorenstein. Moreover, by
Theorem \ref{hilbseries}, the Hilbert series is
$$\mH_{K[\mC]}(t) = \frac{1 + \sum_{i = 1}^{5} 4\, t^i + t^6}{(1-t)^2}$$
and by Theorem \ref{hilbfunction} the Hilbert function is
$$H\!F_{K[\mC]}(s)= \left\{
\begin{array}{ll} \binom{s+2}{2} + 3\binom{s+1}{2} &\ \  {\rm if}\ 0 \leq s < 5, \\
22\, s - 44 & \ \ {\rm if}\  s \geq 5,
\end{array} \right.$$ being $\alpha = 5$ the regularity of the Hilbert function, because $k = 3 < n-1 =4$ (see Remark \ref{regularityHilbFunction}).
\end{example}

\begin{example}
Consider the set $\{m_1,m_2,m_3,m_4,m_5\}$ with $m_1 := 4,\, m_2 :=
5,\, m_3:=6,\, m_4 := 7$ and $m_5 := 8$. We observe that $m_1 =
0 \cdot 4 + 4$, then $q = 0$ and $r = 4$, and we have that $v = 1$, $\alpha = 1$ and that $k = 1$. Then
$${\rm reg}(K[\mC]) = \lceil (8 - 1) / (5 - 1) \rceil = 2.$$ Since $m_1 - 1 = 0(n-1) + 3$,
the  Cohen-Macaulay type is
 $\tau = 3$. Moreover, the Hilbert series and Hilbert function are respectively
\begin{center}$\mH_{K[\mC]}(t) = \frac{1 +  4\, t + 3\, t^2}{(1-t)^2} \ \ {\rm and}\ \
H\!F_{K[\mC]}(s)= \left\{
\begin{array}{ll} 1 &\ \  {\rm if}\ s = 0, \\
8\, s - 2 & \ \ {\rm if}\  s \geq 1,
\end{array} \right.$\end{center} being $\alpha = 1$ the regularity of the Hilbert function.
\end{example}

\section{Generalized arithmetic sequences}

This section concerns the study of $K[\mC]$ when $\m1mn$ is a
generalized arithmetic sequence and $n \geq 3$. We discard the case
$n = 2$ because two integers always form an arithmetic sequence and,
thus, the study of $K[\mC]$  was covered by the previous section.

\medskip
The first goal of this section is to prove that whenever $\m1mn$ is
a generalized arithmetic sequence, then $K[\mC]$ is Cohen-Macaulay
if and only if $h = 1$. We prove this as a direct consequence of
Theorem \ref{notCM_GenArith}, a more general result providing a big
family of projective monomial curves that are not arithmetically
Cohen-Macaulay. More concretely, this result concerns projective
monomial curves $\mC$ defined by  a sequence $\m1mn$ containing a
generalized arithmetic sequence of $l$ terms with $l \geq 3$ and
finishing in $m_n$, i.e., there exist $1 \leq i_1 < \cdots < i_l =
n$ and $h,d \in\Z^+$ such that $m_{i_j} = h m_{i_1}+(j-1)d$ for all
$j\in\{2,\ldots,l\}$.

\begin{theorem}\label{notCM_GenArith}Let $\m1mn$ contain a generalized
arithmetic sequence $m_{i_1} < \cdots < m_{i_l} = m_n$ with $l \geq
3$. If $h>1$  and $h m_{i_1} \notin \{m_1,\ldots,m_n\}$, then
$K[\mC]$ is not Cohen-Macaulay.
\end{theorem}
\begin{proof}We are proving that there exists
a monomial minimal generator of ${\rm in}(I(\mC))$ involving the
variable $x_n$; thus, by \cite[Proposition 2.1]{BerGi01}, we can
conclude that $K[\mC]$ is not Cohen-Macaulay.  We consider the
binomial $f:= x_{i_1}^hx_n-x_{i_2}x_{i_{l-1}}x_{n+1}^{h-1} \in
I(\mC)$ and observe that $\ini(f) = x_{i_1}^h x_n$ because $h
> 1$. If $x_{i_1}^h \notin {\rm in}(I(\mC))$, then there exists
a minimal generator of ${\rm in}(I(\mC))$ involving $x_n$ and
$K[\mC]$ is not Cohen-Macaulay. So assume that $x_{i_1}^h \in {\rm
in}(I(\mC))$, then there exists $x^{\beta} \notin \ini(I(\mC))$ with
$\beta = (\beta_1,\ldots,\beta_{n+1}) \in\N^{n+1}$ such that $g:=
x_{i_1}^h - x^{\beta} \in I(\mC)$. We claim that $\beta_{n+1}<h-1$.
Indeed $x^{\beta}$ has degree $h$ and $x^{\beta} \neq x_n^{h}$, so
$\beta_{n+1} \leq h-1$, and if $\beta_{n+1}=h-1$ we have that
$g=x_{i_1}^h - x_j x_{n+1}^{h-1}\in I(\mC)$, for some
$j\in\{1\ldots,n\},$ which implies that $h m_{i_1}=m_j$, a
contradiction.

Consider now the binomial
$f-x_ng=x^{\beta}x_n-x_{i_{2}}x_{i_{l-1}}x_{n+1}^{h-1}\in I(\mC)$.
Since $\beta_{n+1}<h-1$ we get that $g':= \prod_{j=1}^{n}
x_j^{\beta_j}x_n-x_{i_{2}}x_{i_{l-1}}x_{n+1}^{h-1-\beta_{j+1}}\in
I(\mC)$ and $\ini(g')=\prod_{j=1}^{n}
x_j^{\beta_j}x_n\in\ini(I(\mC))$. Nevertheless, $x^{\beta}\notin
\ini(I(\mC))$, so there exists a minimal generator of $\ini(I(\mC))$
involving $x_n$ and the result follows.
\end{proof}


%

Now, from Theorems  \ref{CcohenMacaulay} and \ref{notCM_GenArith} we
directly derive the following characterization of the Cohen-Macaulay
property for $K[\mC]$ when $\m1mn$ is a
generalized arithmetic sequence and $n \geq 3$.

\begin{corollary}\label{CMcharacteriz_GenArith}
$K[\mC]$ is Cohen-Macaulay if and only if $\m1mn$ is an arithmetic
sequence.
\end{corollary}

Moreover, as a consequence of Corollary \ref{CMcharacteriz_GenArith} and Remark \ref{sistemaMinimal} we
can easily recover \cite[Theorem 6.1]{BerGarMarCons}, where the
complete intersection property for $I(\mC)$ is characterized.

\begin{corollary}\label{IC}Let $\m1mn$ be a generalized arithmetic
sequence of relatively prime integers. Then, $I(\mC)$ is a complete
intersection if and only if $n = 3$, $h = 1$,
 and $m_1$ is even.
\end{corollary}
\begin{proof}$(\Rightarrow)$ Assume that $I(\mC)$ is a complete intersection, then
$K[\mC]$ is Cohen-Macaulay and, by Corollary \ref{CMcharacteriz_GenArith},
$\m1mn$ is an arithmetic sequence. Hence, by Remark
\ref{sistemaMinimal}, we have that $I(\mC)$ is a complete
intersection if and only if $\binom{n-1}{2} + k = n-1$. This
equality can only hold if $n = 3$ and $k = 1$ or, equivalently, if $n = 3$ and $m_1$
even.

$(\Leftarrow)$ A direct application of Remark
\ref{sistemaMinimal} gives the result.
\end{proof}

 In the rest of this section we assume that $\m1mn$ is a generalized arithmetic
sequence of relatively prime integers with $n \geq 3$, $h>1$ and $h$
divides $d$. It turns out that under this hypotheses $I(\mC)$ is
closely related to $I(\mC')$, where $\mC'$ is the projective
monomial curve associated to the arithmetic sequence $m_2 < \cdots <
m_n$ or, equivalently, to the arithmetic sequence of relatively
prime integers $m_2/h < \cdots < m_n/h$, which we have studied in
detail in the previous section. For a sake of convenience we assume
that $I(\mC') \subset K[x_2,\ldots,x_{n+1}]$. The following result
shows how the initial ideals of $I(\mC)$ and $I(\mC')$ with respect
to the degree reverse lexicographic order with $x_1 > \cdots
> x_{n+1}$ are related and will be useful in the rest of this
section.

\begin{proposition}\label{eliminacion} ${\rm in}(I(\mC')) = {\rm in}(I(\mC)) \cap
K[x_2,\ldots,x_{n+1}]$.
\end{proposition}
\begin{proof}It is obvious that ${\rm in}(I(\mC')) \subset {\rm in}(I(\mC)) \cap
K[x_2,\ldots,x_{n+1}]$. To prove the converse we consider $\lambda =
(0,\lambda_2,\ldots,\lambda_{n},\lambda_{n+1}) \in \N^{n+1}$ such
that $x^{\lambda} \in {\rm in}(I(\mC))$ and we assume without loss
of generality that $\lambda_{n+1} = 0$. Then, there exists $\epsilon
= (\epsilon_1,\ldots,\epsilon_{n+1}) \in \N^{n+1}$ such that
$f=x^{\lambda}-x^{\epsilon}\in I(\mC)$ with $\ini(f) = x^{\lambda}$.
We consider two cases:

\emph{Case 1:} If $\epsilon_1 = 0$, then $f \in I(\mC) \cap
K[x_2,\ldots,x_n] = I(\mC').$

\emph{Case 2:} If $\epsilon_1 \neq 0$, then we are proving that
there exists $\mu = (\mu_1,\ldots,\mu_{n+1}) \in \N^{n+1}$ such that
$\mu_1 < \epsilon_1$ and $g := x^{\lambda} - x^{\mu} \in I(\mC)$ with
${\rm in}(g) = x^{\lambda}$. Firstly, we claim that $\epsilon_1 \geq
h$. Indeed, $ \epsilon_1 m_1 = \sum_{i=2}^n (\lambda_i - \epsilon_i)
m_i$, $\gcd\{m_1,h\}=1$ and $h$ divides $m_i$ for all $i \in
\{2,\ldots,n\}$; thus, $\epsilon_1 = h \epsilon_1'$. Moreover, there
exist $j\in\{3,\ldots,n\}$ such that $\epsilon_j>0$. Otherwise, we
have $\epsilon_1 m_1 + \epsilon_2 m_2 = \sum_{i=2}^n\lambda_i m_i$
and $m_1 < m_2 \leq m_i$ for all $i\in\{2,\ldots,n\}$, then
$\sum_{i=2}^n \lambda_i < \epsilon_1 +\epsilon_2$. This implies that
$\lambda_{n+1} >0$ since $f$ homogeneous, a contradiction. So we
take $j\in\{3,\ldots,n\}$ such that $\epsilon_j>0$ and consider
$x^{\mu} = x^{\epsilon} x_1^{-h} x_2 x_{j-1} x_{j}^{-1}
x_{n+1}^{h-1}$ and
 $g :=
x^{\lambda}-x^{\mu}$. One can easily check that $g \in I(\mC)$, that
$x^{\mu} < x^{\epsilon}$, so $\ini(g) = \ini(f)$, and $\mu_1 =
\epsilon_1 - h < \epsilon_1$. Iterating the same argument if
necessary, we get the result.
\end{proof}

\begin{remark}\label{counterexample_prop3.3}When $h$ does not divide
$d$ it is easy to find examples where the conclusion of Proposition
\ref{eliminacion} does not hold. For example, we consider the set
$\{m_1,m_2,m_3,m_4,m_5\}$ with $m_1:=2$, $m_2:=35$, $m_3:=46$,
$m_4:=57$ and $m_5:=68$, where $h=12$ and $d=11$. By Proposition
\ref{gbtheorem}, we have that $\ini(I(\mC'))= \langle
x_2^{23},x_2^{22}x_3,x_3^2,x_3x_4,x_4^2\rangle$. However, with the
help of a computer one can check that $\ini(I(\mC)) \cap
K[x_2,\ldots,x_{6}]=\langle x_2^2,x_3^2,x_3x_4,x_4^2\rangle$.

Even if we consider $\{m_1,m_2,m_3,m_4\}$ with $m_1:=5$, $m_2:= 26$,
$m_3=32$ and $m_4:= 38$, where $h=4$, $d=6$ and then $\gcd\{h,d\}=
h/2 = 2$. We observe that $\ini(I(\mC'))= \langle
x_2^{10},x_2^{9}x_3,x_3^2\rangle\neq\ini(I(\mC)) \cap
K[x_2,\ldots,x_{5}]=\langle x_2^6,x_2^5x_3,x_3^2\rangle$.
\end{remark}

 In order to describe a minimal Gr\"obner basis of $I(\mC)$
we introduce two technical lemmas.

\begin{lemma}\label{lemma_alfa}
Take $p\in\N$ and $s\in\{1,\ldots,n-1\}$ such that $m_1=p(n-1)+s$
and set $\delta := ph + d + h$. Then,
\begin{enumerate}
\item[{\rm (a)}] $\delta m_1 = m_{s+1} + p m_n$,
\item[{\rm (b)}] if $\delta m_1 = m_j + p' m_n$ with $p' \in \N$,
$j \in \{2,\ldots,n\}$, then $j = s+1$ and $p' = p$, and
\item[{\rm (c)}] $\delta = {\rm min}\,\{b\in\Z^+\ \vert\
b\,m_1\in\sum_{i=2}^{n} \N m_i\}$.
\end{enumerate}
\end{lemma}
\begin{proof}
It is straightforward to check $(a)$ and $(b)$. Let us denote $M :=
{\rm min}\,\{b\in\Z^+\ \vert\ b\,m_1\in\sum_{i=2}^{n} \N m_i\}$. We
observe that $h$ divides $M$ because $h$ divides $m_i$ for all $i
\in \{2,\ldots,n\}$ and $\gcd\{m_1,h\} = 1$. Hence, $M = h M'$,
where $M' = {\rm min}\,\{b\in\Z^+\ \vert\ b\,m_1\in\sum_{i=2}^{n} \N
m_i'\}$ and $m_i' = m_i / h$ for all $i \in \{2,\ldots,n\}$.
Applying Lemma \ref{uniquenessLema} to $\{m_1,m_2',\ldots,m_n\}$ we
get that $M' = p + d/h + 1$. Then $M = h M' = \delta$ and $(c)$ is
proved.
\end{proof}

Following the notation of the previous lemma, we state the next
result.

\begin{lemma}\label{lemma_gamma_betas}
Set $\beta_j:= \min\,\{b\in\Z^+ \ \vert\ (jh)m_1 + bm_2
\in\sum_{i=3}^n\N m_i\}$, for all $j\in\{0,\ldots,\delta' - 1\}$,
where $\delta' := \delta/h$, then
\begin{enumerate}
\item[{\rm (a)}] for each $j\in\{0,\ldots,\delta'-1\}$, there exist unique
$\sigma_j\in\{3,\ldots,n\}$ and $\lambda_j\in\N$ such that $jhm_1 +
\beta_jm_2 = m_{\sigma_j} + \lambda_j m_n$, and

\item[{\rm (b)}] setting $\sigma_{\delta'} := s+1$,
$\lambda_{\delta'} := p$ and $\beta_{\delta'} := 0$, then, for all
$j \in \{1,\ldots,\delta'\}$ we have the following:
\begin{itemize}\item[{\rm (b.1)}] if $\sigma_j \neq n$, then $\sigma_{j-1} = \sigma_j +
1$, $\lambda_{j-1}=\lambda_j$ and $\beta_{j-1}= \beta_j +1$;
\item[{\rm (b.2)}] if $\sigma_j = n$, then $\sigma_{j-1} = 3$,
$\lambda_{j-1}=\lambda_j+1$ and $\beta_{j-1}= \beta_j +2$.
\end{itemize}
\end{enumerate}
\end{lemma}
\begin{proof}
By definition of $\beta_j$ we can write
\begin{center}$
jhm_1 + \beta_jm_2 = m_{\sigma_j} +\lambda_j m_n,$
\end{center}
where $\lambda_j\in\N$ and $\sigma_j \in \{3,\ldots,n\}$. Suppose
now that there exist $\sigma_j' \in\{3,\ldots,n\}$ and
$\lambda_j'\in\N$ such that $jhm_1 + \beta_jm_2 = m_{\sigma_j'}
+\lambda_j' m_n.$ We assume without loss of generality that
$\sigma_j \geq \sigma_j'$. Then, $\lambda_j' \geq \lambda_j$ and we
get that $m_{\sigma_j} - m_{\sigma_j'} = (\lambda_j'-\lambda_j)m_n$.
Hence $\lambda_j'=\lambda_j$ and $\sigma_j = \sigma_j'$, and get
$(a)$.

Let us now prove $(b)$. We first consider the case $\sigma_j\neq n$.
Since $jhm_1 + \beta_jm_2 = m_{\sigma_j} + \lambda_j m_n$, then
$(j-1)hm_1 + (\beta_j+1)m_2 = m_{\sigma_j + 1} + \lambda_j m_n$, so
$\beta_{j-1}\leq \beta_j + 1 $. On the other hand, $(j-1)hm_1 +
\beta_{j-1}m_2 = m_{\sigma_{j-1}}+\lambda_{j-1}m_n$, which implies
$jhm_1 + (\beta_{j-1}-1)m_2= m_{\sigma_{j-1}-1} + \lambda_{j-1}m_n$
and we derive $\beta_{j-1}-1\geq \beta_j$. Thus $\beta_j=
\beta_{j-1}-1$, $\sigma_j = \sigma_j-1$ and $\lambda_j =
\lambda_{j-1}$.

Finally, we consider the case $\sigma_j = n$. From the equalities
$jhm_1+\beta_jm_2 = m_{n} + \lambda_j m_n$ and $2 m_2 = m_3 + h
m_1$, we obtain $(j-1)hm_1+(\beta_j + 2)m_2 = m_{3} + (\lambda_j+1)
m_n$ and we deduce that $\beta_{j-1}\leq \beta_j + 2$. On the other
hand, from $(j-1)hm_1+\beta_{j-1}m_2 = m_{\sigma_{j-1}} +
\lambda_{j-1} m_n$, we obtain that
\begin{equation}\label{equality2}
jhm_1+(\beta_{j-1}-1)m_2 = m_{\sigma_{j-1}-1} + \lambda_{j-1} m_n.
\end{equation}
If $\sigma_{j-1} = 3$, then $\beta_{j-1}-2\geq \beta_j$; and if
$\sigma_{j-1} > 3$, $\beta_{j-1}-1\geq \beta_j$. So, $\beta_j +1\leq
\beta_{j-1}\leq \beta_j+2$. Suppose that $\beta_{j-1}=\beta_j + 1$,
 from (\ref{equality2}) we get $jhm_1+\beta_{j}m_2 =
m_{\sigma_{j-1}-1} + \lambda_{j-1} m_n$. Then, from $(a)$, $\sigma_j
= \sigma_{j-1}-1$, a contradiction because $\sigma_j = n$. We
conclude then that $\beta_{j-1} = \beta_j +2$, $\sigma_{j-1}=3$ and
$\lambda_{j-1} = \lambda_j + 1$.
\end{proof}

\begin{remark}\label{remark_gamma_betas}Lemma \ref{lemma_gamma_betas} gives an
algorithmic way
of computing the integers $\beta_j, \sigma_j$ and $\lambda_j$
starting from $\beta_{\delta'} = 0, \lambda_{\delta'} = p$ and
$\sigma_{\delta'}  = s+1$ where $p$ and $s$ are given by Lemma
\ref{lemma_alfa}.  Nevertheless one can easily derive an explicit
formula for $\beta_j, \sigma_j, \lambda_j$ for all
$j\in\{1\ldots,\delta'\}$. Indeed, we have that $\beta_{\delta'-j} =
j +\lfloor(j+s-2)/(n-2)\rfloor$, $\sigma_{\delta'-j}$ is the only integer in
$\{3,\ldots,n\}$ such that $\sigma_{\delta'-j} \equiv s+1+j\ ({\rm
mod}\, n-2)$, and $\lambda_{\delta'-j} = p +
\lfloor(j+s-2)/(n-2)\rfloor$.
\end{remark}

\begin{remark}\label{beta0_alfa_mas_1}
Notice that $\beta_0 = {\rm min}\{b \in \Z^+ \, \vert \, b m_2 \in
\sum_{i = 3}^n \N m_i\}$, so, applying Lemma \ref{uniquenessLema} to
the sequence of relatively prime integers $m_2/h < \cdots < m_n /
h$, we get that $\beta_0 = \alpha + 1$, where $\alpha = \lfloor (m_n
- h) / (nh-2h) \rfloor$. Moreover, if $\mC'$ denotes the projective
monomial curve associated to the arithmetic sequence
$m_2<\cdots<m_n$, we have that ${\rm reg}(K[\mC'])$ equals $\beta_0$
if $\sigma_0 \neq 3$, or $\beta_0 - 1$ if $\sigma_0 = 3$.
\end{remark}

Following the notation of Lemmas \ref{lemma_alfa} and
\ref{lemma_gamma_betas},
 we can describe a minimal Gr\"obner basis of $I(\mC)$ with respect to \lp.

\begin{theorem}\label{GBgeneraliz_h_div_d}
Let $\mG_1 \subset K[x_2,\ldots,x_n,x_{n+1}]$ be a minimal Gr\"obner
basis with res\-pect to \lp\ of $I(\mC')$, where $\mC'$ is the
projective monomial curve associated to the arithmetic sequence
$m_2<\cdots<m_n$. Then,
$$\begin{array}{rclll}
\mG &:=& \mG_1  & \cup & \{x_1^hx_{i}-x_2x_{i-1}x_{n+1}^{h-1}\ \vert\ 3\leq i\leq n\} \\
    & & &\cup& \{x_1^{jh}x_2^{\beta_j}-x_{\sigma_j}x_n^{\lambda_j}x_{n+1}^{j(h-1)+ (d/h)}\ \vert\ 1\leq j\leq\delta/h\},
\end{array}$$
is a minimal Gr\"obner basis of $I(\mC)$ with respect to \lp.
Moreover, $\mG$ is a minimal set of generators of $I(\mC)$.
\end{theorem}

\begin{proof}
We set $h_i := x_1^h x_{i}-x_2 x_{i-1} x_{n+1}^{h-1} \in I(\mC)$ for
all $i\in\{3,\ldots,n\}$ and $p_j := x_1^{jh}
x_2^{\beta_j}-x_{\sigma_j}x_n^{\lambda_j}x_{n+1}^{j(h-1) + (d/h)}$
for all $j\in\{1,\ldots,\delta/h\}$. By Lemma
\ref{lemma_gamma_betas}, we have that $p_j \in I(\mC)$. Moreover,
since $I(\mC')\subset I(\mC)$, we clearly have that $\mG_1\subset
I(\mC)$.

%
%
Let us prove now that $\ini(I(\mC))\subset \ini(\mG)$, where
$\ini(\mG)$ is the ideal generated by the initial terms of binomials
in $\mG$ with respect to the degree reverse lexicographic order. We
consider a binomial $f=x^{\lambda}-x^{\epsilon}\in I(\mC)$ with
$\lambda = (\lambda_1,\ldots,\lambda_{n+1}), \epsilon =
(\epsilon_1,\ldots,\epsilon_{n+1}) \in \N^{n+1}$, ${\rm
gcd}\,\{x^{\lambda},x^{\epsilon}\} = 1$ and $\ini(f) = x^{\lambda}$.
We consider different cases:

\emph{Case 1:} If $\lambda_1 = 0$, then $\ini(f)\in \ini(I(\mC))\cap
k[x_2,\ldots,x_{n+1}]$ and, by Proposition \ref{eliminacion}, we get
that $\ini(f) \in \ini(\mG_1)\subset\ini(\mG)$ .

\emph{Case 2:} If $\lambda_1 > 0$, then $h$ divides $\lambda_1$
because $\lambda_1m_1 = \sum_{i = 2}^n (\epsilon_i - \lambda_i)
m_i$, $h$ divides $m_i$ for all $i\in\{2,\ldots,n\}$ and
$\gcd\{m_1,h\}=1$. If there exists $i \in \{3,\ldots,n\}$ such that
$\lambda_i > 0$, then $\ini(h_i)$ divides $x^{\lambda}$ and we
deduce that $x^{\lambda} \in \ini(\mG)$. So, it only remains to
consider when $f =
x_1^{\lambda_1}x_2^{\lambda_2}-x_2^{\epsilon_2}\cdots
x_{n+1}^{\epsilon_{n+1}}$. If $\lambda_1 \geq \delta$, then
$\ini(p_{\delta/h})$ divides $x^{\lambda}$ by Lemma
\ref{lemma_alfa}. If $\lambda_1 < \delta$, then $\lambda_1 m_1 +
\lambda_2 m_2 = \epsilon_3m_3+\cdots+\epsilon_{n}m_{n}$ and, taking
$j := (\delta - \lambda_1) / h$ we have that $\delta_2 \geq
\beta_{j}$ by the definition of $\beta_j$. Thus $\ini(p_{j})$
divides $\ini(f)$ and the proof is complete.

To prove that $\mG$ is a minimal set of generators, it suffices to
prove it when $\mG_1$ is the Gr\"obner basis of $I(\mC')$ described
in Theorem \ref{gbtheorem}. By Remark \ref{sistemaMinimal}, $\mG_1$
is a minimal set of generators of $I(\mC')$. Since all the elements
of degree $2$ in $\mG$ belong to $\mG_1$ it follows that the
elements of degree $2$ are not redundant. Let us take $f \in \mG$ of
degree $\geq 3$, if $f = h_i$ or $f = p_i$, then $\ini(f)$ is not
divided by any other of the monomials appearing in the elements of
$\mG \setminus \{f\}$ and, hence, $f$ is not redundant. Assume now
that $\ini(f) = x_2^{\alpha} x_i$ and \begin{equation}\label{eq:f} f
= \sum_{g \in \mathcal G \setminus \{f\}} h_g\, g,\end{equation}
where $h_g = 0$ or ${\rm deg}(h_g) = {\rm deg}(f) - {\rm deg}(g)$.
We consider the homomorphism $\tau$ induced by $x_1 \mapsto 0$,
$x_{n+1} \mapsto 0$ and $x_i \mapsto x_i$ for all $i \in
\{2,\ldots,n\}$. Applying $\tau$ to (\ref{eq:f}), we derive that
there exists a linear form $l$ such that $x_2^{\alpha}l \in L$,
where $L$ is the binomial prime ideal $\langle x_j x_k - x_{j-1}
x_{k+1}\, \vert \, 3 \leq j < k \leq n-1\rangle$. Hence $l \in L$, a
contradiction. Thus, $\mG$ is a minimal set of generators of
$I(\mC)$.
\end{proof}

\begin{remark}
Putting together Theorems \ref{gbtheorem} and
\ref{GBgeneraliz_h_div_d}, we obtain a minimal Gr\"obner basis of
$I(\mC)$ with respect to the \lp. Observe that if we consider $h=1$
in this basis, the set $\mG$ obtained is not a minimal Gr\"obner
basis of $I(\mC)$.
\end{remark}

\begin{remark}\label{GBnotMinSetGen}
If we remove the assumption that $h$ divides $d$, in general, a
minimal Gr\"obner basis of $I(\mC)$ with respect to \lp \ may not be a
minimal system of generators of $I(\mC)$. Indeed, if we consider the
first example in Remark \ref{counterexample_prop3.3}, one has that a
minimal Gr\"obner basis of $I(\mC)$ has $9$ elements, whereas the
first Betti number is $8$.
\end{remark}

In the next result, we exploit the description of the minimal
Gr\"obner basis of $I(\mC)$ obtained in Theorem
\ref{GBgeneraliz_h_div_d} to get the irreducible decomposition of
$\ini(I(\mC))$.

\begin{proposition}\label{irredDecGeneralized}
Let $\delta$ and $\beta_j$ be the as in Lemmas \ref{lemma_alfa} and
\ref{lemma_gamma_betas} and let $\{\mathfrak q_i\}_{i \in I}$ be the
irreducible components of $I(\mC')$. Then,
$$\ini(I(\mC))= \bigcap_{i \in I}\langle
x_1^h,\mathfrak{q}_i\rangle\cap\left[
\bigcap_{j=2}^{\delta/h}\langle
x_1^{jh},x_2^{\beta_{j-1}},x_3,\ldots,x_n\rangle\right].$$ is the
irredundant irreducible decomposition of $\ini(I(\mC))$.
%
%
%
%
%
%
%
%
%
%
\end{proposition}
\begin{proof} We set
$A_i = \langle x_1^h,\mathfrak{q}_i\rangle$ and $B_j:=\langle
x_i^{jh},x_2^{\beta_{j-1}},x_3,\ldots,x_n\rangle$. From Theorem
\ref{GBgeneraliz_h_div_d} we obtain that
\begin{center}$\begin{array}{rcl} \ini(I(\mC)) & = & \ini(I(\mC')) + \langle x_1^hx_i\ \vert\ 3\leq i\leq n \rangle
+ \langle x_1^{jh}x_2^{\beta_j}\ \vert\ 1\leq j\leq \delta /
h\rangle
\end{array}$\end{center} is a minimal system of generators of
$\ini(I(\mC))$ with respect to the degree reverse lexicographic
order with $x_1>\cdots > x_{n+1}$. Then, it is easy to check that
$\ini(I(\mC))\subset A_i,B_j$, for all $i\in I$ and for all $j \in
\{2,\ldots,\delta/h\}$. In order to see the reverse inclusion, we
consider $l\notin\ini(I(\mC))$, we distinguish two possibilities:
\begin{itemize}
\item[{\rm (a)}] $l=x_1^ax_ix_{n+1}^c$ with $0 \leq a<h$, $i\in\{2,\ldots,n\}$ and $c\geq
0$, or

\item[{\rm (b)}] $l=x_1^{a}x_2^bx_{n+1}^c$ with $h \leq jh\leq a< \delta$, $b<\beta_j$ and $c\geq 0$.
\end{itemize}
In (a), setting $l' = x_i x_{n+1}^c$ we have that $l'
\notin\ini(I(\mC'))$ and, as a consequence, $l' \notin
\mathfrak{q}_i$ for some $i\in I$, thus $l\notin A_i$.  If (b)
holds, then $l\notin B_{j+1}$.

It only remains to show that there is no redundant ideal in the
intersection. The same argument as in Theorem \ref{irredDec} applies
here to derive that $A_i$ is not redundant for all $i \in I$.
Finally,  $B_j$ is no redundant for $j\in\{2,\ldots,\delta/h\}$
because $x_1^{jh-1} x_2^{\beta_{j-1}-1}\notin\ini(I(\mC))$, but
$x_{1}^{jh-1}\in B_k$ for all $k<j$ and $x_2^{\beta_{j-1}-1}\in B_k$
for all $k>j$.
\end{proof}

Once we have the irreducible decomposition of $\ini(I(\mC))$ we can
obtain the value of the regularity.

\begin{theorem}\label{reg_generalizec_h_div_d}
Let $\m1mn$ be a generalized arithmetic sequence with
$\gcd\,\{m_1,d\}=1$, $h>1$ and $h$ divides $d$. Set $\delta :=
\lfloor (m_1 - 1) / (n-1)\rfloor h+d+h$, then
$${\rm reg}(K[\mC])=\left\{ \begin{array}{ll} \delta - 1 & {\rm if }\ n-1\ {\rm does\ not\ divide }\ m_1, \\
\delta & {\rm if}\ n-1\ {\rm divides }\ m_1.
\end{array} \right.$$
\end{theorem}
\begin{proof}
Observe that, by Proposition \ref{irredDecGeneralized}, we have
$$\ini(I(\mC))= \bigcap_{i \in I} \langle x_1^h,\mathfrak{q}_i\rangle\cap\left[ \bigcap_{j=2}^{\delta/h}\langle x_1^{jh},x_2^{\beta_{j-1}},x_3,\ldots,x_n\rangle\right],$$
where $\{\mathfrak q_i\}_{i \in I}$ are the irreducible components
of $\ini(I(\mC'))$. A direct application of Lemma \ref{regNT} yields
$${\rm reg}(K[\mC])= \max\,\{h - 1 +{\rm reg}\,(K[\mC']),
\max_{j\in\{2,\ldots,\delta/h\}}\{jh+\beta_{j-1}-2\}\}.$$  Moreover,
it is clear that
$\max_{j\in\{2,\ldots,\delta/h\}}\{jh+\beta_{j-1}-2\}= \delta +
\beta_{\delta/h-1}-2$ since $h\geq 2$ and $\beta_{j-1}-\beta_j\leq
2$ for all $j\in\{2,\ldots,\delta/h\}$. Thus we get that ${\rm
reg}(K[\mC])=\max\{h-1+{\rm reg}(K[\mC']),\delta+\beta_{\delta/h
-1}-2\}$. Our objective is to prove that this maximum equals
$\delta+\beta_{\delta/h -1}-2$ or, equivalently, that
\begin{equation}\label{objetivo} {\rm reg}(K[\mC']) - \beta_{(\delta/h)-1} \leq
\delta-h-1.\end{equation}

In the proof we use the following three facts:
\begin{itemize}
\item[(i)] By Theorem \ref{regfor} and Remark
\ref{beta0_alfa_mas_1}, we have that ${\rm reg}(K[\mC']) \in
\{\beta_0-1, \beta_0\}$. \item[(ii)] By Lemma
\ref{lemma_gamma_betas}, we have that
\begin{center}$ \beta_0 -
\beta_{(\delta/h) - 1} = \sum_{i=1}^{(\delta/h)-1} (\beta_{i-1} -
\beta_{i}) \leq 2 ((\delta/h) - 1) = 2(\delta/h) -
2.\label{gamma_mas_1} $\end{center} \item[(iii)] Since $h$ divides
$\delta$ and $m_1 \notin \sum_{i=2}^n \N m_i$, from Lemma
\ref{lemma_alfa}(c) we have that $\delta \geq 2h$. \end{itemize}

We divide the proof in four cases: \begin{itemize}\item[(a)] $h
\geq 3$, \item[(b)] ${\rm reg}(K[\mC']) = \beta_0 - 1$, \item[(c)]
$\beta_0 - \beta_{(\delta/h)-1} \leq 2 (\delta/h) - 3$, and
\item[(d)] $h = 2$, ${\rm reg}(K[\mC']) = \beta_0$ and $\beta_0 - \beta_{(\delta/h)-1} = 2 (\delta/h) -
2$.
\end{itemize}

If (a) holds, then by (iii) we have that $h (h-1) \leq 2 h  (h-2)
\leq (h-2) \delta$, then $2(\delta/h) - 2 \leq \delta - h -1$.
Therefore, by (i) and (ii), we get that ${\rm reg}(K[\mC']) -
\beta_{(\delta/h) - 1} \leq \beta_0 - \beta_{(\delta/h) - 1} \leq
2(\delta/h) - 2  \leq \delta - h - 1$, and we conclude
(\ref{objetivo}). If (b) or (c) holds, then $2(\delta/h) - 2 \leq
\delta - h$. Therefore, by (i) and (ii), in both cases we get that
${\rm reg}(K[\mC']) - \beta_{(\delta/h) - 1}  \leq 2(\delta/h) - 3
\leq \delta - h - 1$, and we conclude (\ref{objetivo}). Let us prove
that (d) is not possible. By contradiction, assume that (d) holds.
Since ${\rm reg}(K[\mC']) = \beta_0$, by Theorem \ref{regfor} we get
that $n \geq 4$. We claim that $\delta = 4$. Indeed, setting
$\Delta_i := \beta_{i-1} - \beta_{i}$ for all $i \in
\{1,\ldots,(\delta/2) -1\}$ we observe that $\Delta_i \in \{1,2\}$
and that $\sum_{i = 1}^{(\delta/2) - 1} \Delta_i = \beta_0 -
\beta_{(\delta/2) - 1} = 2(\delta/2) - 2$, so $\Delta_i = 2$ for all
$i \in \{1,\ldots, (\delta/2) - 1\}$. Nevertheless $n \geq 4$, hence
from Lemma \ref{lemma_gamma_betas} we have that two consecutive
values of $\Delta_i$ cannot be both $2$, so we conclude that $\delta
= 4$ and $\beta_0 - \beta_1= 2$, so $\beta_0 = 3$ and $\sigma_0 =
3$. However, this implies that ${\rm reg}(K[\mC']) = \beta_0-1$ (see
Remark \ref{beta0_alfa_mas_1}), a contradiction.

To conclude the result it suffices to observe that, by Lemma
\ref{lemma_gamma_betas}, $\beta_{(\delta/h) - 1} \in \{1,2\}$ and
that it is $2$ if and only if $n-1$ divides $m_1$

\end{proof}

In the following result we make use of a result of \cite{BerGi00} to
prove that in this setting, the regularity is always attained at the
last step a minimal graded free resolution of $K[\mC]$.

\begin{corollary}\label{ultimopaso}
${\rm reg}(K[\mC])$ is attained at the last step of a minimal graded
free resolution of $K[\mC]$.
\end{corollary}
\begin{proof}
Let $F$ be the set of $\gamma = (\gamma_1,\ldots,\gamma_{n-1},0,0)
\in\N^{n+1}$ such that
$x^{\gamma}\in\ini(I(\mC))\vert_{x_n=x_{n+1}=1}$ and $x^{\gamma}
\notin \ini(I(\mC))\vert_{x_n=x_{n+1}=0}$. According to
\cite[Corollary 2.11]{BerGi00}, a sufficient condition for the
regularity to be attained at the last step of the resolution is that
${\rm reg}(K[\mC]) = \max\{\sum_{i = 1}^{n-1} \gamma_i \ \vert\
\gamma\in F\}$. From Theorem \ref{GBgeneraliz_h_div_d}, we get that
$$\begin{array}{rcl}
\ini(I(\mC))\vert_{x_n=x_{n+1}=1} & = & \ini(I(\mC'))\vert_{x_n=x_{n+1}=1} + \langle x_1^h \rangle, \\
\ini(I(\mC))\vert_{x_n=x_{n+1}=0} & = &
\ini(I(\mC'))\vert_{x_n=x_{n+1}=0} + \langle \{x_1^hx_i\ \vert\
3\leq i\leq n-1\}\rangle + \\ & & \langle\{x_1^{ih}x_2^{\beta_i}\
\vert\ 1\leq i \leq \delta/h\}\rangle,
\end{array}$$
where $\mC'$ is the projective monomial curve associated to the
arithmetic sequence $\{m_2,\ldots,m_n\}$. From Theorem
\ref{gbtheorem}, we have that $ \ini(I(\mC'))\vert_{x_n=x_{n+1}=1} =
\ini(I(\mC'))\vert_{x_n=x_{n+1}=0}$; and then we obtain that  $F =
\{(\gamma_1,\gamma_2,0,\ldots,0)\ \vert\ jh\leq\gamma_1<(j+1)h,\
0\leq\gamma_2<\beta_j,\ {\rm for\ all }\ j\in\{1,\ldots,\delta/h
-1\}\}$. Since $h\geq 2$ and $\beta_{j}\leq\beta_{j+1}+2$ we
conclude that $\max\{\sum_{i=1}^{n-1} \gamma_i \ \vert\ \gamma\in
F\} = \delta+\beta_{\delta/h-1}-2$, which equals ${\rm
reg}(K[\mC])$, from where the result follows.

\end{proof}

In the last part of this section we provide formulas for the Hilbert
series and the Hilbert function of $K[\mC]$ in this setting.

\begin{theorem}\label{HilbertSeriesGeneralized}
The Hilbert series of $K[\mC]$ is
$$ \mH_{K[\mC]}(t)= (1+\cdots+t^{h-1})\mH_{K[\mC']}(t)+ \frac{\sum_{j=h}^{\delta-1}t^j-
\left(\sum_{j=0}^{h-1}t^j\right) \left(
\sum_{i=1}^{(\delta/h)-1}t^{ih+\beta_i}\right)}{(1-t)^2},  $$ where
$\mH_{K[\mC']}(t)$ is the Hilbert series of $K[\mC']$ with $\mC'$ is
the projective monomial curve associated to the arithmetic sequence
$m_2<\cdots<m_n$.
\end{theorem}
\begin{proof}
By Proposition \ref{hilbseriesgeneral}, setting $J := \ini(I(\mC))$
we have that
$$
\begin{array}{ccl} (1-t)^2\, \mH_{K[\mC]}(t) & = & \sum_{x^{\gamma} \notin J + \langle
x_n, x_{n+1}\rangle} t^{\, \vert \gamma \vert} - t \sum_{x^{\gamma}
\notin J + \langle x_{n+1} \rangle \atop x_n x^{\gamma} \in J} t^{\,
\vert \gamma \vert} \\ & = &  \sum_{x^{\gamma} \notin J + \langle
x_n, x_{n+1}\rangle \atop \gamma_1 < h} t^{\, \vert \gamma \vert} +
S, \end{array}$$ where  $S := \sum_{x^{\gamma} \notin J + \langle
x_n, x_{n+1}\rangle \atop \gamma_1 \geq h} t^{\, \vert \gamma \vert}
- t \sum_{x^{\gamma} \notin J + \langle x_{n+1} \rangle \atop x_n
x^{\gamma} \in J} t^{\, \vert \gamma \vert}$. We observe in Theorem
\ref{GBgeneraliz_h_div_d} that $x_1^hx_n$ is the unique minimal
generator of $J$ where $x_n$ appears, so for all $\lambda =
(\lambda_1,\ldots,\lambda_{n+1}) \in \N^{m+1}$ we have that
\begin{itemize}\item[(a)] $x^{\lambda} \notin\ini(I(\mC))$ with
$\lambda_1 < h$ if and only if $x_2^{\lambda_2} \cdots
x_{n+1}^{\lambda_{n+1}}\notin\ini(I(\mC'))$, \item[(b)] if
$x^{\lambda} \notin\ini(I(\mC))$, then $x_n x^{\lambda} \in
\ini(I(\mC))$ if and only if $\lambda_1 \geq h$.\end{itemize} From
(a) we derive that
$$(1 + t + \cdots + t^{h-1}) \mH_{K[\mC']}(t) = \frac{\sum_{x^{\gamma} \notin J + \langle
x_n, x_{n+1}\rangle \atop \gamma_1 < h} t^{\, \vert \gamma
\vert}}{(1-t)^2}.$$ From (b) we derive that
$$\begin{array}{rcl}S &  = &  (1-t) \sum_{x^{\gamma} \notin J + \langle x_{n+1} \rangle,\,
\atop \gamma_1 \geq h} t^{\,\vert \gamma \vert}\\ & = &
\sum_{j=h}^{\delta-1}t^j-
\left(\sum_{j=0}^{h-1}t^j\right)\left(\sum_{i=1}^{\delta/h-1}t^{ih+\beta_i}\right),
\end{array} $$ where the last equality is directly deduced from Theorem
\ref{GBgeneraliz_h_div_d} and we are done.

\end{proof}

As a direct consequence of the previous result, we obtain a simpler
expression of the Hilbert series of $K[\mC]$ when $n=3$.

\begin{corollary}\label{HilbertSeries_n3}
If $n=3$, the Hilbert series  of $K[\mC]$ is $\mH_{K[\mC]}(t)=
\frac{h(t)}{(1-t)^2},$ where
\begin{enumerate}
\item $h(t) = 1+2t+\sum_{j=2}^h3t^j+t^{h+1}-t^{2h}$, if $\delta = 2h$ and $m_1$ is odd,

\item $h(t) = 1+2t+3t^2+\sum_{j=3}^h4t^j+3t^{h+1}+t^{h+2}-t^{2h}-t^{2h+1}$,  if $\delta = 2h$ and $m_1$ is even,

\item $ h(t)  =  1+2t+\sum_{j=2}^{(\delta - 4)/2}3t^j-((\delta - 6)/2)t^{\delta-1}
 - ((\delta - 2)/2)t^{\delta}$, if $\delta > 4$ and $h=2$,

\item $h(t)  =  \sum_{j=0}^{h-1}(j+1)t^j+\sum_{j=h}^{(m_3/h)-1}(h+1)t^j+
\sum_{j=0}^{h-3}(h-j)t^{(m_3/h)+j}+t^{(m_3/h)+h-2}\\  -
\sum_{j=2}^{(\delta/h)-1}(t^{jh+\beta_j}+t^{jh+\beta_j+1})-t^{\delta}$,
if $\delta > 2h$, $h\geq 3$ and $m_1$ is odd, \smallskip

\item $h(t) = \sum_{j=0}^{h-1}(j+1)t^j+\sum_{j=h}^{(m_3/h)-1}(h+1)t^j+\sum_{j=0}^{h-3}(h-j)t^{(m_3/h)+j}+t^{(m_3/h)+h-2}\\
 \ \ \ -\sum_{j=2}^{\delta/h}(t^{jh+\beta_j}+t^{jh+\beta_j+1})$, if $\delta > 2h$, $h\geq 3$ and $m_1$ is even.
\end{enumerate}

\end{corollary}

\medskip
\begin{theorem}\label{HilbFunctionGeneralized}
The Hilbert function of $K[\mC]$ is
$$H\!F_{K[\mC]}(s) = \sum_{i=0}^{h-1}H\!F_{K[\mC']}(s-i) + (n-3)\Delta_{s} + \Delta_{s+1},$$
where $\Delta_s := \# \{(a,b)\in \N^2 \, \vert\, a+b < s\ {\rm and}\
b <\beta_j,\,{\rm with}\ j:= \lfloor a/h\rfloor \geq 1\}$, $\mathcal
C'$ is the projective monomial curve associated to the arithmetic
sequence $m_2<\cdots<m_n$ and $H\!F_{K[\mC']}(a):=0$ whenever
$a< 0$.
\end{theorem}
\begin{proof}
The Hilbert function of $K[\mC]$ coincides with that of
$R/\ini(I(\mC))$. Thus $H\!F_{K[\mC]}(s) = \#
\{x^{\gamma}\notin\ini(I(\mC))\,\vert\,\deg(x^\gamma)=s\}$. We write
$$\begin{array}{rcl}H\!F_{K[\mC]}(s) & = & \# \{x^{\gamma}\notin\ini(I(\mC))\,\vert\,\deg(x^{\gamma})=s,\,\gamma_1<h\}\\ & + &
\#
\{x^{\gamma}\notin\ini(I(\mC))\,\vert\,\deg(x^{\gamma})=s,\,\gamma_1\geq
h\}.\end{array}$$ From Theorem \ref{GBgeneraliz_h_div_d}, we easily
deduce that $\{x^{\gamma}\notin\ini(I(\mC))\,\vert\,
\deg(x^{\gamma})=s,\,\gamma_1<h\} =
\{x_1^{\gamma_1}x^{\gamma'}\notin\ini(I(\mC))\,\vert\,\gamma' =
(0,\gamma_2,\ldots,\gamma_{n+1}), \deg(x^{\gamma'})+\gamma_1=s, \
{\rm and}\ \gamma_1<h\} = \{x^{\gamma'}\notin\ini(I(\mC'))\,\vert\,
s-h+1\leq\deg(x^{\gamma'})\leq s,\,\gamma_1'=0\}$. Thus
$$H\!F_{K[\mC]}(s) = \sum_{i=0}^{h-1}H\!F_{K[\mC']}(s-i) +
\#
\{x^{\gamma}\notin\ini(I(\mC))\,\vert\,\deg(x^{\gamma})=s,\,\gamma_1\geq
h\}.$$ \noindent Moreover, from Theorem \ref{GBgeneraliz_h_div_d} we
get that $\#
\{x^{\gamma}\notin\ini(I(\mC))\,\vert\,\deg(x^{\gamma})=s,\,\gamma_1\geq
h\}  = \#
\{x_1^{\gamma_1}x_2^{\gamma_2}x_i^{\epsilon_i}x_{n+1}^{\gamma_{n+1}}\notin\ini(I(\mC)))\,\vert\,\gamma_1+\gamma_2+\epsilon_i+\gamma_{n+1}=
s, \gamma_1\geq h, \epsilon_i\in\{0,1\}\ {\rm and}\
i\in\{3,\ldots,n-1\} \} =  (n-3)\Delta_{s} + \Delta_{s+1}, $ which
completes the proof.
\end{proof}

Moreover, we can also describe the Hilbert polynomial
$P_{K[\mC]}(s)$ of $K[\mC]$ in terms of the Hilbert polynomial
$P_{K[\mC']}(s) = (m_n/h) s + \gamma$ of $K[\mC']$ described in
Theorem \ref{hilbfunction} (see also Remark
\ref{regularityHilbFunction}).

\begin{corollary}\label{HilbertPolynomial_generalized}
The Hilbert polynomial of $K[\mC]$ is
$$P_{K[\mC]}(s) = m_n s+ m_n\frac{h-1}{2} + h \gamma + (n-2)h\sum_{i=0}^{\delta/h-1}\beta_i.$$
\end{corollary}
\begin{proof}We set $\Delta := h\sum_{i=0}^{\delta/h -
1}\beta_i$. To get the result it suffices to observe that $\Delta_s
= \Delta$ for all $s >> 0$  and that $P_{K[\mC]}(s) = \sum_{i =
0}^{h-1} P_{K[\mC']}(s-i) + \Delta$ by Theorem
\ref{HilbFunctionGeneralized}.
\end{proof}

\medskip
Let us illustrate the results of this section with an example.
\begin{example}\label{generalized_example}
Consider the set $\{m_1,m_2,m_3,m_4,m_5,m_6\}$ with $m_1 := 7$,
$m_2:= 30, m_3 :=39, m_4 := 48$, $m_5:=57$ and $m_6 := 66$. We
observe that $m_1,\ldots,m_6$ are relatively prime, $h = 3$ and $d =
9$, so $h$ divides $d$. We compute $\delta=\lfloor (7-1) /
(6-1)\rfloor 3 + 9 +3 = 15$ and observe that $n-1$ does not divide
$m_1$. Then, by Theorem \ref{reg_generalizec_h_div_d}, we get that
${\rm reg}(K[\mC])=14.$

Moreover, from the expression $m_1 = 1\cdot 5 + 2$, we get that
$p=1$ and $s=2$. Following remark \ref{remark_gamma_betas} we get
that $\beta_1 = 4 + \lfloor (4+2-2)/4\rfloor = 5$, $\beta_2 = 3
+\lfloor (3+2-2)/4\rfloor = 3$, $\beta_3 = 2 +\lfloor
(2+2-2)/4\rfloor = 2$ and $\beta_4 = 1 +\lfloor (1+2-2)/4\rfloor =
1$. By
 Theorem \ref{HilbertSeriesGeneralized}
we obtain that
$$\mH_{K[\mC]}(t) = (1+t+t^2)\mH_{K[\mC']}(t) +
\frac{\sum_{j=3}^{14}t^j -
(1+t+t^2)(t^8+t^9+t^{11}+t^{13})}{(1-t)^2}.$$ We observe that the
sequence $10 < 13 < 16 < 19 < 22$ defines the same projective
monomial curve as the sequence $30 < 39 < 48 < 57 < 66$. Then, from
Example \ref{ejemlo1Aritmeticas} we get that
$\mH_{K[\mC']}(t)=\frac{1+\sum_{i=1}^54t^i + t^6}{(1-t)^2}$. Hence,
we conclude that $\mH_{K[\mC]}(t)$ equals
$$\frac{1+5t+9t^2+13(t^3+t^4+t^5)+10t^6+6t^7+t^8-t^9
-t^{10}-t^{11}-t^{13}-t^{15}}{(1-t)^2}.$$
\end{example}

\section{Koszulness}

In this section we characterize when $K[\mC]$ is a Koszul algebra
when $\m1mn$ is a generalized arithmetic sequence (Theorem
\ref{koszulGeneralizedArithmetic}), when $n = 3$ (Theorem
\ref{Koszulcodim2}) and when $n = 4$ (Theorem \ref{Koszulcodim3}).

We recall (see, e.g., \cite{ConcaDeNegriRossi2013}) that the
following well known implications (i) $\Rightarrow$ (ii)
$\Rightarrow$ (iii) hold:
\begin{itemize}
\item[(i)] $I(\mC)$ possesses a quadratic Gr\"obner basis \item[(ii)]
$K[\mC]$ is Koszul
\item[(iii)] $I(\mC)$ is generated by quadrics. \end{itemize}
These implications will be frequently used in this section.


\begin{theorem}\label{koszulGeneralizedArithmetic}Let
$\m1mn$ be a generalized arithmetic sequence of relatively prime
integers. Then, $K[\mC]$ is a Koszul algebra if and only if $\m1mn$
are consecutive numbers and $n
> m_1$.
\end{theorem}

\medskip

For proving this result we make use of the following lemma.

\begin{lemma}\label{cuadricas}Take $i \in
\{1,\ldots,n-1\}$; if $I(\mC)$ is generated by quadrics,  then $2
m_i = m_j + \mu m_l$, where $j,l \in \{1,\ldots,n\} \setminus \{i\}$
and $\mu \in \{0,1\}$.
\end{lemma}
\begin{proof} Set $w \in \Z^+$ the minimum integer
such that $f:= x_i^w - \prod_{j \in \{1,\ldots,n+1\} \atop j \neq i}
x_j^{\alpha_j} \in I(\mC)$ for some
$\alpha_1,\ldots,\alpha_{i-1},\alpha_{i+1},\ldots, \alpha_{n+1} \in
\N$. It turns out that there exists a minimal set of generators of
$I(\mC)$ containing $f$ (see, e.g., \cite[Lemma 3.1.(b)]{BerGarMar14});
hence, $f$ is a quadric. This implies that $f = x_i^2 - x_j x_l$ for
some $j \in \{1,\ldots,n\} \setminus \{i\}, l \in \{1,\ldots,n+1\}
\setminus \{i\}$. Thus we conclude that $2 m_i = m_j$ if $k = n+1$,
or $2 m_i = m_j + m_l$ otherwise.

\end{proof}

\begin{proof}[Proof of Theorem \ref{koszulGeneralizedArithmetic}]
$(\Rightarrow)$ Since $I(\mC)$ is generated by quadrics, by Lemma
\ref{cuadricas} we get that $2 m_1
> m_2$ and, thus, $h = 1$. Moreover, if $n \leq m_1$, then $2 m_1 \neq
m_j$ for all $j \in \{1,\ldots,n-1\}$ and $2 m_1 \neq m_j + m_l$ for
all $j, l \in \{2,\ldots,n\}$, then again by Lemma
\ref{cuadricas} we derive that $n > m_1$.

 $(\Leftarrow)$ We observe
that $\m1mn$ is an arithmetic sequence, $d = 1$, and that
$\alpha = \lceil m_1 / (n-1) \rceil + d - 1 = 1$. As a consequence,
the Gr\"obner basis $\mG$ in Theorem \ref{gbtheorem} only consists
of quadrics and we conclude that $K[\mC]$ is a Koszul algebra.
\end{proof}

%

In the rest of this section we aim at listing all projective
monomial curves $\mC$ such that its corresponding algebra $K[\mC]$
is Koszul when $n = 3$ (Theorem \ref{Koszulcodim2}) and $n = 4$
(Theorem \ref{Koszulcodim3}).

\begin{theorem}\label{Koszulcodim2}If $n = 3$ and $\gcd\{m_1,m_2,m_3\} = 1$;
then, $K[\mC]$ is a Koszul algebra if and only if $\{m_1,m_2,m_3\}$
is $\{1,2,3\}, \{1,2,4\}$ or $\{2,3,4\}$.
\end{theorem}

\begin{theorem}\label{Koszulcodim3}
If $n = 4$ and $\gcd\{m_1,m_2,m_3,m_4\}=1$; then, $K[\mC]$ is a
Koszul algebra if and only if $\{m_1,m_2,m_3,m_4\}$ is one of the
following sequences
\begin{center} $\{1,2,3,4\},\{1,2,3,5\},\{1,2,3,6\},\{1,2,4,6\},\{1,2,4,8\},\{2,3,4,5\},\{2,3,4,6\},$ \\ $\{2,3,4,8\},\{2,4,5,6\},\{2,4,5,8\},\{3,4,5,6\},\{3,4,6,8\},
\{4,5,6,8\} \ {\rm or} \ \{4,6,7,8\}. $ \end{center}
\end{theorem}

To obtain these results we first introduce two lemmas concerning a
new projective curve $\mC_m$, defined for $m < m_n$ . More
precisely, $\mC_m$ is the projective monomial curve in
$\mathbb{P}^{n+1}_{K}$ parametrically defined by
$$x_1=s^{m_1}t^{m_n-m_1},\ldots,x_{n-1}=s^{m_{n-1}}t^{m_n-m_{n-1}},x_{n}=s^{m_n},x_{n+1}=t^{m_n}, y=s^{m} t^{m_n-m},$$
we denote by $I(\mC_m) \subset K[x_1,\ldots,x_{n+1},y]$ the
vanishing ideal of $\mC_m$. We denote by $B_m$ the integer $\gcd\{m_1,\ldots,m_n\} /
\gcd\{m_1,\ldots,m_n,m\}$.

\begin{lemma}\label{lema2}
Take $d \in \Z^+$ and set $m_i := 2^{i-1} d$ for all
$i\in\{1,\ldots,n\}$; then, $K[\mC]$ is a Koszul algebra. Moreover,
if $B_m = 2$ and $2 m = m_i + m_j$ for some $i, j \in
\{1\ldots,n\}$, then $K[\mC_m]$ is also a Koszul algebra.
\end{lemma}
\begin{proof}
We have that $m_i/d = 2^{i-1}$ for all $i\in\{1,\ldots,n\}$. Thus,
by \cite[Theorem 5.3 and Corollary 5.8]{BerGarMar14}, $I(\mC)$ is a
complete intersection and it is generated by quadrics, which implies
that  $K[\mC]$ is Koszul (see,  e.g., \cite[Corollary 8 and Theorem
10]{ConcaDeNegriRossi2013}).

We observe that $B_m (m,m_n-m) = 2 (m, m_n - m) = (m_i,m_n -
m_i)+(m_j,m_n-m_j)$. Therefore, by \cite[Lemma 2.1 and Proposition
2.3]{BerGarMar14}, we get that $I(\mC_m)=I(\mC) + \langle y^2 -
x_ix_j\rangle$ and $I(\mC_m)$ is also a complete intersection
generated by quadrics, then $K[\mC_m]$ is a Koszul algebra.
\end{proof}

\begin{lemma}\label{lema3}
We set $d:=\gcd\{m_1,\ldots,m_n\}$. If $m_1/d,\ldots,m_n/d$ are
consecutive numbers, $n > m_1$, $B_m=2$ and $2 m = m_i + \mu
 m_j$ for some $i,j \in \{1,\ldots,n\}$ and $\mu \in \{0,1\}$, then $K[\mC_m]$ is a Koszul algebra.
\end{lemma}
\begin{proof} By Theorem \ref{koszulGeneralizedArithmetic} and its proof, the
Gr\"obner basis with respect to the degree reverse lexicographic
order $\mG$ provided by Theorem \ref{gbtheorem} consists of
quadrics. Moreover, we observe that $B_m (m, m_n - m) = 2(m, m_n-m)
= (m_i,m_n-m_i)+(m_j,m_n-m_j)$ if $\mu = 1$ or $B_m (m, m_n - m) =
2(m, m_n-m) = (m_i,m_n-m_i)+(0,m_n)$ if $\mu = 0$; thus by
\cite[Lemma 2.1 and Proposition 2.3]{BerGarMar14} we get that
$I(\mC_m) =I(\mC) + \langle y^2-x_ix_l \rangle$, where $l = j$ if
$\mu = 1$ or $l = n+1$ if $\mu = 0$. Let us prove that  $\mG_m :=
\mG \cup\{y^2-x_ix_l\}$ is a Gr\"obner basis for $I(\mC_m)$ with
respect to the monomial ordering $>_y$ in $K[x_1,\ldots,x_{n+1},y]$
given by:
\begin{center}
$\prod_{i=1}^{n+1}x_i^{\alpha_i}y^{\alpha}>_{y}\prod_{i=1}^{n+1}x_i^{\beta_i}y^{\beta}$ $\Longleftrightarrow$ $\left\{ \begin{array}{l}  \alpha > \beta, {\rm \ or} \\
\alpha = \beta \  {\rm and} \
\prod_{i=1}^{n+1}x_i^{\alpha_i}>_{dp}\prod_{i=1}^{n+1}x_i^{\beta_i},
\end{array}\right.$
\end{center}
where $>_{dp}$ is the degree reverse lexicographic order in
$K[x_1,\ldots,x_{n+1}]$. Indeed we take $f =
\prod_{i=1}^{n+1}x_i^{\alpha_i}y^{\alpha}
-\prod_{i=1}^{n+1}x_i^{\beta_i}y^{\beta}\in I(\mC_m)$, where
$\alpha_i,\alpha, \beta_i, \beta\in\N$ and suppose that
$\ini_{>_{y}}(f) = \prod_{i=1}^{n+1}x_i^{\alpha_i}y^{\alpha}$, we
consider three different cases.

\emph{Case 1:} If $\alpha = 0$, then $\beta = 0$ and
$f=\prod_{i=1}^{n+1}x_i^{\alpha_i}
-\prod_{i=1}^{n+1}x_i^{\beta_i}\in I(\mC)$. Then there exist $g\in
\mG \subset \mG_m$ such that $\ini_{>_{y}}(g)=\ini_{>_{gr}}(g)\mid
\ini_{>_{y}}(f) $.

\emph{Case 2:} If $\alpha = 1$, then $\beta\in\{0,1\}$. If $\beta =
1$, then
$f/y=\prod_{i=1}^{n+1}x_i^{\alpha_i}-\prod_{i=1}^{n+1}x_i^{\beta_i}\in
I(\mC)$. Thus there exists $g\in\mG\subset\mG_m$ such that
$\ini_{>_{y}}(g)=\ini_{>_{dp}}(g)\mid \ini_{>_{dp}}(f/y) \mid
\ini_{>_{y}}(f)$. If $\beta = 0$, then $f =
\prod_{i=1}^{n+1}x_i^{\alpha_i}y - \prod_{i=1}^{n+1}x_i^{\beta_i}\in
I(\mC_m)$, and it follows that
\begin{center}$\sum_{i=1}^{n+1}\alpha_i(m_i,m_n-m_i) + (m,m_n-m) = \sum_{i=1}^{n+1}\beta_i(m_i,m_n-m_i),$\end{center}
and then $m \in \sum_{i\in\{1,\ldots,n\}} \Z m_i$, which implies
$B_m = 1$, a contradiction.

\emph{Case 3:} If $\alpha\geq 2$, then
$\ini_{>_{y}}(y^2-x_ix_j)=y^2\mid \ini_{>_{y}}(f)$.

Thus $\mG_m = \mG \cup\{y^2-x_ix_j\}$ is a quadratic Gr\"obner basis
of $I(\mC_m)$ with respect to $>_y$; therefore $K[\mC_m]$ is a
Koszul algebra.

\end{proof}

\begin{proof}[Proof of Theorem \ref{Koszulcodim2}]
$(\Rightarrow)$ Since  $I(\mC)$ is generated by quadrics, by Lemma
\ref{cuadricas} we get that either $2 m_1 \in \{m_2,m_3\}$ and that
$2 m_2 \in \{m_3, m_1 + m_3\}$. Since $\gcd\{m_1,m_2,m_3\} = 1$ we
easily derive that $\{m_1,m_2,m_3\}$ equals $\{1,2,4\}, \{1,2,3\}$
or $\{2,3,4\}$.

\noindent $(\Leftarrow)$ For $\{m_1,m_2,m_3\} = \{1,2,4\}$ then
$K[\mC]$ is Koszul by Lemma \ref{lema2}. If $\{m_1,m_2,m_3\}$ equals
$\{1,2,3\}$ or $\{2,3,4\}$, we conclude that $K[\mC]$ is Koszul by
Theorem \ref{koszulGeneralizedArithmetic}.

\end{proof}

\begin{proof}[Proof of Theorem \ref{Koszulcodim3}]
$(\Rightarrow)$ Since $I(\mC)$ is generated by quadrics, by Lemma
\ref{cuadricas}, we have that $2 m_1 \in \{m_2, m_3, m_4\},\ 2m_2
\in \{m_3, m_4, m_1+m_3, m_1+m_4\}$ and that $2 m_3 \in \{m_4,
m_1+m_4, \, m_2 + m_4\}$. A case by case analysis of all the
possible values of $2m_1, 2m_2$ and $2m_3$ together with the fact
that $\gcd\{m_1,m_2,m_3,m_4\} = 1$ proves this implication.

\noindent $(\Leftarrow)$ For  $\{1,2,3,4\}, \, \{2,3,4,5\}$ and
$\{3,4,5,6\}$, we get that $K[\mC]$ is a Koszul algebra by Theorem
\ref{koszulGeneralizedArithmetic}. By Lemma \ref{lema2}, we get that
$K[\mC]$ is Koszul for the sets $\{1,2,4,8\}$, $\{2,4,5,8\}$ and
$\{2,3,4,8\}$. Moreover, for $\{1,2,4,6\}$, $\{2,4,5,6\},
\{2,3,4,6\}, \{3,4,6,8\}, \{4,5,6,8\}$ and $\{4,6,7,8\}$ we prove
the implication applying Lemma \ref{lema3}. Finally, for
$\{1,2,3,6\}$ and $\{1,2,3,5\}$, a direct computation yields that
the Gr\"obner basis of $I(\mC)$ with respect to the degree reverse
lexicographic order is formed by quadrics, so $K[\mC]$ is Koszul as
well.

\end{proof}

\section*{Acknowledgments}

The authors were partially supported by Ministerio de
Econom\'ia y Competitividad,  Spain (MTM2010-20279-C02-02, MTM2013-40775-P). The second
author was also partially supported by CajaCanarias.


\begin{thebibliography}{0}

\bibitem{BaMum93}{D.~Bayer, D.~Mumford, What can be computed in algebraic geometry? {\it Computational
Algebraic Geometry and Commutative Algebra (Cortona 1991)}, 1--48, {\it Cambridge Univ. Press, Cambridge}, 1993.}

\bibitem{BerGarMar14}{I.~Bermejo, I.~Garc\'{\i}a-Marco, Complete intersections in
simplicial toric varieties. {\it  J. Symbolic Comput.} {\bf 68} (2015), part 1,
265--286.}

\bibitem{BerGarMarCons}{I.~Bermejo, I.~Garc\'{\i}a-Marco, Complete intersections in certain
affine and projective monomial curves, Bull Braz Math Soc, New
Series {\bf 45(4)}, 2014, 1-26.}

\bibitem{BerGi00}{I.~Bermejo, P.~Gimenez, On Castelnuovo-Mumford regularity of projective curves.
{\it Proc. Amer. Math. Soc.} {\bf 128} (2000), no. 5, 1293--1299.}

\bibitem{BerGi01}{I.~Bermejo, P.~Gimenez, Computing the Castelnuovo-Mumford regularity of some subeschemes of $\mathbb{P}_{K}^n$ using quotients
of monomial ideals. {\it J. Pure Appl. Algebra} {\bf 164} (2001), 23--33.}

\bibitem{BerGi06}{I.~Bermejo, P.~Gimenez, Saturation and Castelnuovo-Mumford regularity. {\it J. Algebra} {\bf 303} (2006), 592--617.}

\bibitem{CavaNiesi83}{M.P.~Cavaliere, G.~Niesi, On monomial curves and Cohen-Macaulay type. {\it Manuscripta Math.}
{\bf 42} (1983), 147--159.}

\bibitem{ConcaDeNegriRossi2013}{A.~Conca, E.~De Negri, M.E.~Rossi, Koszul
algebras and regularity. {\it Commutative algebra}, 285--315, {\it Springer,
New York}, 2013.}

\bibitem{EiGo84}{D.~Eisenbud, S.~Goto, Linear free resolutions and minimal multiplicities. {\it J. Algebra}
{\bf 88} (1984), 84--133.}

\bibitem{LiPatilRoberts12}{P.~Li, D.E.~Patil, L.G.~Roberts, Bases and ideal generators for projective monomial curves. {\it Comm. Algebra} {\bf 40} (2012),no. 1,173--191}

\bibitem{MolinelliTamone1995} S.~Molinelli, G.~Tamone, On the Hilbert function of certain rings of monomial curves. {\it J. Pure Appl. Algebra} {\bf 101} (1995), no. 2, 191--206

\bibitem{Stu96}{B.~Sturmfels, Gr\"obner bases and convex
polytopes. University Lecture Series {\bf 8}. {\it American Mathematical
Society, Providence}, 1996.}

\bibitem{Vi01}{R.H.~Villarreal, Monomial Algebras, Monographs and Textbooks in Pure and Applied
Mathematics {\bf 238}. {\it Marcel Dekker, New York}, 2001.}

\bibitem{Vi02}{R.H.~Villarreal, Monomial Algebras, Second Edition, Monographs and Research Notes in Mathematics.
{\it Chapman and Hall/CRC}, 2015.}



\end{thebibliography}
\end{document}